\documentclass[12pt]{article}

\usepackage{graphicx}    
\usepackage{amsmath} 
\usepackage{xcolor}
\usepackage{amssymb, amsthm} 
\usepackage{booktabs}
\usepackage{hyperref}    
\usepackage{geometry}    
\usepackage{caption}     
\usepackage{cite}
\usepackage{algorithm}
\usepackage{algorithmicx}
\usepackage{algpseudocode}
\usepackage{pgfplots}
\usepackage{bm}
\usepackage{float}
\usetikzlibrary{patterns,positioning}
\usepackage{tikz}
\usetikzlibrary{decorations.markings}
\usetikzlibrary{matrix}
\usepackage{pgfplots}
\usetikzlibrary{shapes, arrows.meta, patterns}
\usepackage{mathrsfs}

\usepackage{subcaption}

\newcommand{\bfa}[1]{\boldsymbol{#1}} 			
			
\newcommand{\bfeps}{\boldsymbol{\epsilon}}

\newcommand{\Sym}{\text{Sym}}   			%
\newcommand{\curl}{\text{curl}}   				%
\newcommand{\tr}{\text{tr}}       				%

\newcommand{\jump}[1]{[\![#1]\!]}

\DeclareMathAlphabet{\mathpzc}{OT1}{pzc}{m}{it}

\newcommand{\bfu}{\boldsymbol{u}}

\newcommand{\bfF}{\boldsymbol{F}}	
\newcommand{\bfE}{\boldsymbol{E}}	
\newcommand{\bfC}{\boldsymbol{C}}
\newcommand{\bfB}{\boldsymbol{B}}	
\newcommand{\bfx}{\boldsymbol{x}}	
\newcommand{\bfX}{\boldsymbol{X}}	
	
\newcommand{\bfT}{\boldsymbol{T}}		
\newcommand{\bfI}{\boldsymbol{I}}	 

\usepackage{authblk}

\usepackage{etoolbox}

\newcommand\norm[1]{\left\lVert#1\right\rVert}

\newtheorem{Theorem}{Theorem}[section]
\newtheorem{remark}{Remark}
\newtheorem*{cf}{Strong formulation}
\newtheorem*{cwf}{Continuous weak formulation}

\title{ $hp$-adaptive finite element simulation of a static anti-plane shear crack in a nonlinear strain-limiting elastic solid }

\author[1,*]{S. M. Mallikarjunaiah}
\author[2]{Pavithra Venkatachalapthy}

\affil[1]{Department of Mathematics \& Statistics, Texas A\&M University-Corpus Christi, TX- 78412, USA}
\affil[2]{Department of Mathematics \& Statistics, Texas Tech University,  Lubbock, TX, 79409-1042, USA}
\affil[*]{Corresponding author}
\affil[ ]{\textit{E-mail addresses:} \texttt{M.Muddamallappa@tamucc.edu} (S. M. Mallikarjunaiah), \texttt{pavvenka@ttu.edu} (Pavithra V.)}

\date{}

\begin{document}

\maketitle  

\begin{abstract}
An $hp$-adaptive continuous Galerkin finite element method is developed to analyze a static anti-plane shear crack embedded in a nonlinear, strain-limiting elastic body. The geometrically linear material is described by a constitutive law relating stress and strain that is algebraically nonlinear. Such a formulation is advantageous as it regularizes the stress and strain fields in the neighborhood of a crack-tip, thereby circumventing the non-physical strain singularities inherent to linear elastic fracture mechanics. In this investigation, the constitutive relation utilized is \textit{uniformly bounded}, \textit{monotone}, \textit{coercive}, and \textit{Lipschitz continuous}, ensuring the well-posedness of the mathematical model. The governing equation, derived from the balance of linear momentum coupled with the nonlinear constitutive relationship, is formulated as a second-order quasi-linear elliptic partial differential equation. For a body with an edge crack, this governing equation is augmented with a classical traction-free boundary condition on the crack faces. An $hp$-adaptive finite element scheme is proposed for the numerical approximation of the resulting boundary value problem. The adaptive strategy is driven by a dual-component error estimation scheme: mesh refinement ($h$-adaptivity) is guided by a residual-based a posteriori error indicator of the \textit{Kelly type}, while the local polynomial degree ($p$-adaptivity) is adjusted based on an estimator of the local solution regularity. The performance, accuracy, and convergence characteristics of the proposed method are demonstrated through numerical experiments. The structure of the regularized crack-tip fields is examined for various modeling parameters. Furthermore, the presented framework establishes a robust foundation for extension to more complex and computationally demanding problems, including quasi-static and dynamic crack propagation in brittle materials.

\end{abstract}

\vspace{.1in}


\section{Introduction}

The Finite Element Method (FEM) has long been a cornerstone of computational mechanics, providing a powerful framework for solving complex boundary value problems (BVPs) across various engineering and scientific disciplines \cite{zienkiewicz2005finite}. Traditional FEM approaches typically rely on either $h$-refinement, which involves subdividing the mesh into smaller elements, or $p$-refinement, which increases the polynomial order of the basis functions within elements. While effective for many problems, pure $h$-refinement often exhibits slow algebraic convergence rates, especially for issues with localized features or singularities, necessitating a huge number of elements to achieve the desired accuracy. Conversely, pure $p$-refinement, though offering faster convergence for smooth solutions, can be less adaptable to complex geometries or regions with abrupt changes in solution gradients. To overcome these limitations, the $hp$-adaptive finite element method emerged as a sophisticated and highly efficient alternative \cite{babuska1981p,babuvska1986basic}. The fundamental principle of $hp$-FEM is to strategically combine both $h$-refinement and $p$-refinement within a single adaptive scheme. This dual-adaptive capability allows the method to optimally distribute computational resources, refining the mesh in regions where the solution exhibits high gradients or singularities (e.g., crack-tips) while simultaneously increasing the polynomial degree in areas where the solution is smoother, thereby capturing intricate details with remarkable precision.

A primary advantage of $hp$-FEM is its capacity to deliver exponential rates of convergence for a wide class of problems, provided the solution possesses a certain level of analytical regularity \cite{schwab1998p,babuvska1990p,szabo1986mesh}. This superior convergence behavior translates directly into significant computational savings: to achieve a prescribed level of accuracy, an $hp$-adaptive approach typically requires far fewer degrees of freedom compared to purely $h$- or $p$-refinement strategies. This efficiency is particularly valuable for computationally intensive nonlinear problems, where each iteration incurs a high cost. Moreover, $hp$-FEM provides a robust framework for error estimation and control, enabling the adaptive algorithm to guide the refinement process towards an optimal mesh and polynomial distribution. 

For problems involving stress concentrators or singularities, such as cracks in elastic solids, $hp$-FEM is particularly well-suited \cite{kato1993adaptive}. Classical linear elastic fracture mechanics predicts singular stress and strain fields at the crack-tip, which pose significant challenges for numerical methods. By employing h-refinement near the singularity to capture the steep gradients and p-refinement away from it to accurately represent the smoother global solution, hp-FEM can effectively resolve these challenging features \cite{andersson1992reliable}. This adaptive capability ensures that the highly localized and rapidly changing fields around the crack-tip are accurately approximated, which is crucial for reliable predictions of fracture behavior. 

The accurate characterization of stress and strain fields in the vicinity of geometric discontinuities, such as notches, slits, holes, and material defects, constitutes a foundational challenge in both practical engineering design and theoretical solid mechanics. Historically, the analysis of these critical stress concentrations has predominantly relied upon the tenets of linearized elasticity theory, as articulated in seminal works \cite{Inglis1913,lin1980singular,love2013treatise,murakami1993stress}. A well-documented limitation of this classical framework is its inherent prediction of physically unrealistic, unbounded strain singularities at the tips of such discontinuities. This pathological behavior stems directly from the first-order linear approximation of finite deformation inherent in the theory. The imperative to develop constitutive representations that yield physically congruent material responses has thus driven extensive research into more sophisticated models, including those presented in \cite{gurtin1975,sendova2010,MalliPhD2015,ferguson2015,zemlyanova2012,WaltonMalli2016,rajagopal2011modeling,gou2015modeling}, alongside the application of advanced numerical techniques like various collocation methods \cite{yadav2024gegenbauer,gazonas2023numerical}. Nevertheless, a persistent hurdle remains in balancing enhanced model fidelity with computational tractability and the crucial aspect of experimental validation \cite{broberg1999}. Many proposed model enhancements frequently introduce substantial computational overheads or present formidable challenges in empirical verification. Furthermore, the application of linear elastic fracture mechanics (LEFM) to the modeling of crack initiation and propagation is subject to intrinsic limitations that demand careful consideration. Beyond the established strain singularity at the crack-tip, these restrictions encompass the prediction of a physically implausible blunt crack-opening profile and the potential for crack-face interpenetration, particularly problematic in bimaterial interfaces. It is noteworthy that the issue of crack-tip singularity often persists even within certain nonlinear elasticity frameworks, as demonstrated by examples such as \cite{knowles1983large} and the bell constraint model in \cite{tarantino1997nonlinear}. Consequently, a salient question arises concerning the capacity of specific algebraic nonlinear models to effectively regulate the crack-tip strain singularity, even when singular stresses are still present.

A significant advancement in the theory of elasticity, extending beyond the conventional Cauchy and Green formulations, has been pioneered by Rajagopal and his collaborators through a series of foundational publications \cite{rajagopal2003implicit,rajagopal2007elasticity,rajagopal2007response,rajagopal2009class,rajagopal2011non,rajagopal2011conspectus,rajagopal2016novel,rajagopal2018note}. This comprehensive body of work, often referred to as \textit{Rajagopal's theory of elasticity}, introduces implicit constitutive models that are firmly rooted in a rigorous thermodynamic framework. Within this theory, the mechanical response of an elastic body---defined as a material incapable of energy dissipation---is precisely characterized by implicit relations between the Cauchy stress and deformation gradient tensors \cite{bustamante2018nonlinear}. A particularly compelling attribute of Rajagopal's methodology is its potential to yield a hierarchical structure of  ``explicit" nonlinear relationships, wherein linearized strain can be expressed as a nonlinear function of Cauchy stress. Crucially, a distinct subclass of these implicit models facilitates the representation of linearized strain as a uniformly bounded function throughout the entire material domain, even under conditions of substantial applied stress. This inherent 'limiting strain' property renders these models exceptionally well-suited for the investigation of crack and fracture behavior in brittle materials \cite{rajagopal2011modeling,gou2015modeling,mallikarjunaiah2015direct,MalliPhD2015}, and holds considerable promise for extension to quasi-static and dynamic crack evolution analyses. Utilizing these strain-limiting models, numerous studies have re-examined and provided new insights into classical elasticity problems \cite{kulvait2013,rajagopal2018bodies,bulivcek2014elastic,erbay2015traveling,itou2018states,zhu2016nonlinear,csengul2018viscoelasticity,itou2017contacting,yoon2022CNSNS,yoon2022MAM}. In essence, strain-limiting constitutive models offer a remarkably versatile framework for elucidating the mechanical behavior of a diverse spectrum of materials, with particular advantages in the comprehensive analysis of crack and fracture phenomena. Recent investigations \cite{lee2022finite,yoon2021quasi} have further demonstrated that the formulation of quasi-static crack evolution problems within this strain-limiting theoretical framework results in a rich array of complex crack patterns, including observations of increased crack-tip propagation velocities.

This investigation focuses on the behavior of a static crack subjected to anti-plane shear loading, embedded within an elastic solid governed by a nonlinear strain-limiting modeling framework. This constitutive formulation isn't just a mathematical construct; it's derived directly from \textit{Rajagopal's theory of elasticity}, ensuring a physically consistent framework that accurately reflects the real-world behavior of materials. This approach allows us to define the response of elastic solids in a far more general setting than traditional models. This means it can account for complexities like large deformations, anisotropic behavior, or material nonlinearities often overlooked by simpler theories, providing a more robust and realistic description of material mechanics. The combination of the linear momentum balance equation with this algebraically nonlinear constitutive law results in a quasi-linear elliptic BVP. Given the inherent intractability of obtaining analytical solutions for such nonlinear partial differential equations, a robust numerical methodology is indispensable. To accurately approximate the solution, we employ the $hp$-adaptive finite element method.  The method's ability to achieve exponential rates of convergence makes it a powerful tool for resolving the intricate details of the stress, strain, and strain-energy density fields in the vicinity of the crack. The inherent nonlinearities of the system are effectively managed through the implementation of Newton's iterative algorithm. The convergence of the numerical solution is rigorously demonstrated by the progressive reduction of the residual at each iteration, ensuring the reliability of our results. Our analysis reveals several key insights into the crack-tip fields under anti-plane shear: specifically, stress concentration is observed to intensify with increasing values of the parameter $\alpha$, while an inverse trend is noted for the parameter $\beta$. Furthermore, the crack-tip strain values exhibit a growth pattern distinct from that predicted by linear elastic models, and the strain-energy density consistently remains highest near the crack-tip along the reference line. This foundational study provides critical insights into the behavior of cracks in nonlinear strain-limiting materials and lays the groundwork for future extensions, including the examination of dynamic crack propagation or more complex loading conditions.

The paper is organized as follows: Section~\ref{intro} introduces Rajagopal's implicit theories of elasticity, which describe the response of elastic materials. The anti-plane shear problem, the focus of this investigation, is also detailed in the same section. The boundary value problem and the existence of its solution are explained in Section~\ref{sec:bvp_existence}. The hp-adaptive finite element method used to simulate the static anti-plane crack is described in detail in Section~\ref{sec:numerical_method}. The results and discussions are outlined in Section~\ref{rd}, and the paper concludes with conclusions and future work in Section~\ref{conclusions}.

\section{Introduction to Rajagopal's theory of elasticity}\label{intro}

In solid mechanics, the mechanical behavior of the elastic material is modeled by the following system of governing equations:

\begin{itemize}
\item An equilibrium equation derived from the balance of linear and angular momentum.
\item A strain-displacement relation based on the infinitesimal strain assumption.
\item A strain compatibility condition to ensure a continuous displacement field.
\item A  linear or nonlinear constitutive law that explicitly or implicitly relates stress and strain.
\end{itemize}

Classical Cauchy elasticity assumes an explicit constitutive relationship where stress is a direct function of a strain measure. To overcome the limitations inherent in such models, particularly for materials exhibiting complex nonlinear responses, Rajagopal introduced a more general implicit theory \cite{rajagopal2003implicit,rajagopal2007elasticity,rajagopal2011non,rajagopal2011conspectus,rajagopal2014nonlinear,rajagopal2007response}. In this framework, stress and a suitable strain measure are related through an implicit algebraic equation. This approach is particularly adept at describing phenomena such as strain-limiting behavior, where strain is bounded regardless of the applied stress.

This investigation leverages a specific class of these implicit models to study the behavior of an isotropic elastic body containing geometric discontinuities. Our ultimate aim is to formulate and implement a convergent numerical method capable of solving the BVP associated with this advanced constitutive model.

\subsection{Kinematic preliminaries}
Let $\mathcal{B}$ be an elastic body occupying a two-dimensional domain, with $\bfX$ and $\bfx$ denoting points in the reference and current configurations. The fundamental kinematic quantities include the displacement $\bfu$, the deformation gradient $\bfF$, the right and left Cauchy-Green stretch tensors ($\bfC$, $\bfB$), and the Lagrange strain $\bfE$:
\begin{subequations}
\begin{align}
    \bfF &:= \bfI + \nabla \bfu, \\
    \bfC &:= \bfF^{\mathrm{T}}\bfF, \\
    \bfB &:= \bfF\bfF^{\mathrm{T}}, \\
    \bfE &:= \frac{1}{2} \left( \bfC - \bfI \right).
\end{align}
\end{subequations}
This work operates within the realm of infinitesimal strain theory, which assumes the displacement gradient is small, i.e.,
\begin{equation}
 \max_{\bfX \in \mathcal{B}} \| \nabla_X \bfu \| \ll 1.
 \end{equation}
 The notation $\| \cdot \|$ refers to the vector norm when applied to a vector and the Frobenius norm when applied to a tensor. This crucial assumption allows for the linearization of the above quantities, leading to the relations $\bfB \approx \bfI + 2 \bfeps$ and $\bfE \approx \bfeps$, where $\bfeps$ is the symmetric infinitesimal strain tensor:
\begin{equation}
    \bfeps := \frac{1}{2} \left( \nabla \bfu + \nabla \bfu^{\mathrm{T}}\right).
\end{equation}

\subsection{Constitutive model formulation}
Rajagopal's general implicit framework posits a relation of the form $\mathcal{F}(\bfT, \bfB) = \bfa{0}$, where $\bfT$ is the symmetric Cauchy stress tensor. For this relation to be objective (frame-indifferent), $\mathcal{F}$ must be an isotropic tensor function. We focus on a notable subclass of these models where the stretch is an explicit function of the stress:
\begin{equation}\label{eq:SL_model}
    \bfB := \mathcal{G}(\bfT).
\end{equation}
If the function $\mathcal{G}$ is bounded such that $\sup_{\bfT \in \Sym(2)} \| \mathcal{G}(\bfT) \| \leq M$ for some constant $M>0$, the model is termed \textit{strain-limiting} \cite{MalliPhD2015,mallikarjunaiah2015direct}.

By applying the infinitesimal strain approximation $\bfB \approx \bfI + 2 \bfeps$ to \eqref{eq:SL_model}, we arrive at a constitutive law relating the small strain tensor $\bfeps$ to the Cauchy stress $\bfT$. For an isotropic material, this relationship can be expressed using a representation theorem involving the invariants of $\bfT$. This leads to the specific form used in our work:
\begin{equation} \label{eq:final_constitutive}
     \bfeps = \Psi_{0}\left( \tr(\bfT), \| \bfT \| \right) \bfI + \Psi_{1}\left( \| \bfT \| \right) \bfT,
\end{equation}
where $\Psi_{0}$ and $\Psi_{1}$ are scalar-valued functions, and to ensure consistency in the stress-free state, we require $\Psi_{0}(0, \cdot) = 0$.

\subsection{The governing boundary value problem}
The complete BVP combines the constitutive relation \eqref{eq:final_constitutive} with the fundamental principles of mechanics. The final system of equations to be solved is:
\begin{subequations}\label{eq:full_governing_system_expanded}
\begin{align}
    -\nabla \cdot \bfT &= \bfa{0}, \quad \text{and} \quad \bfT = \bfT^\mathrm{T}, \tag{\ref{eq:full_governing_system_expanded}a} \\
    \bfeps &= \Psi_{0}\left( \tr(\bfT), \| \bfT \| \right) \bfI + \Psi_{1}\left( \| \bfT \| \right) \bfT, \tag{\ref{eq:full_governing_system_expanded}b} \\
    \curl(\curl(\bfeps)) &= \bfa{0}, \tag{\ref{eq:full_governing_system_expanded}c} \\
    \bfeps &= \frac{1}{2} \left( \nabla \bfu + \nabla \bfu^{\mathrm{T}} \right). \tag{\ref{eq:full_governing_system_expanded}d}
\end{align}
\end{subequations}
This system comprises the balance of linear and angular momentum (a), the nonlinear strain-limiting constitutive law (b), the Saint-Venant compatibility condition (c), and the linearized strain-displacement relation (d).

\subsection{Anti-plane shear (mode-III) formulation}

This work presents an efficient numerical method for simulating the static response of a strain-limiting elastic body containing a V-notch under anti-plane shear. The investigation is novel in its application of algebraically nonlinear constitutive relations to this class of fracture mechanics problems. The analysis is specialized to the case of Mode-III loading, where the displacement field $\bfa{u}$ has only one non-zero component, which is a function of the in-plane coordinates:
\begin{equation}\label{eq:disp_vector}
    \bfa{u}(x, y) = (0, 0, w(x, y)).
\end{equation}
Consequently, the only non-zero components of the infinitesimal strain tensor $\bfeps$ are $\epsilon_{13}$ and $\epsilon_{23}$, and the corresponding non-zero stress components of the Cauchy stress tensor $\bfT$ are $T_{13}$ and $T_{23}$.

For anti-plane shear, the stress tensor is deviatoric, meaning its trace is zero ($\tr(\bfT) = 0$). Recalling the general constitutive form from Equation~\eqref{eq:full_governing_system_expanded} and the property that $\Psi_{0}(0, \cdot) = 0$, the first term vanishes identically. The constitutive law, therefore, naturally simplifies to:
\begin{equation}\label{eqn:constitutive_mode3}
    \bfeps = \Psi_{1}(\| \bfT \|) \bfT.
\end{equation}
This relation connects the strain components directly to the stress components:
\begin{equation}\label{eq:strain_components}
    \epsilon_{13} = \Psi_{1}(\| \bfT \|) T_{13}, \quad \epsilon_{23} = \Psi_{1}(\| \bfT \|) T_{23}.
\end{equation}

To ensure the balance of linear momentum is satisfied, we introduce an Airy stress function $\Phi(x, y)$. The stress components are defined as the curl of this potential:
\begin{equation}
    T_{13} = \frac{\partial \Phi}{\partial y}, \quad T_{23} = -\frac{\partial \Phi}{\partial x}.
\end{equation}
This definition identically satisfies the equilibrium equation in the absence of body forces:
\begin{equation}
    \frac{\partial T_{13}}{\partial x} + \frac{\partial T_{23}}{\partial y} = \frac{\partial}{\partial x}\left(\frac{\partial \Phi}{\partial y}\right) + \frac{\partial}{\partial y}\left(-\frac{\partial \Phi}{\partial x}\right) = 0.
\end{equation}
The strain compatibility condition in  Equation~\eqref{eq:full_governing_system_expanded} simplifies significantly for this problem. Since the strain components depend only on $x$ and $y$, it reduces to a single scalar equation:
\begin{equation}\label{eq:strain_compatibility_mode3}
    \frac{\partial \epsilon_{13}}{\partial y} - \frac{\partial \epsilon_{23}}{\partial x} = 0.
\end{equation}
Substituting the constitutive relations \eqref{eq:strain_components} into the compatibility condition \eqref{eq:strain_compatibility_mode3}, and subsequently using the Airy stress function definitions, we arrive at a single governing equation for $\Phi$. 
This yields the following second-order quasilinear elliptic partial differential equation (PDE) for the stress potential $\Phi$:
\begin{equation}\label{pde:nlin1}
    - \nabla \cdot \left( \Psi_{1} ( \| \nabla \Phi \| ) \nabla \Phi \right) = 0.
\end{equation}
For the subsequent development of the BVP, we adopt the following specific form for $\Psi_{1}$, which models strain-limiting behavior:
\begin{equation}\label{eq:Psi_specific}
    \Psi_{1}(\|\bfT\|) = \frac{1}{2\mu (1 +  \left( \beta\|\bfT\| \right)^{\alpha})^{1/\alpha}},
\end{equation}
where $\mu$ is the classical shear modulus. The governing PDE \eqref{pde:nlin1} now takes the concrete form:
\begin{equation} \label{pde:mech}
    - \nabla \cdot \left( \frac{\nabla \Phi}{2\mu (1 + \left( \beta \,  \|\nabla \Phi\| \right)^\alpha )^{1/\alpha}} \right) = 0.
\end{equation}
The constants $\alpha > 0$ and $\beta \geq 0$ are fundamental modeling parameters; variations of these parameters give rise to a diverse class of constitutive models.  Specifically, the parameter $\alpha$ generally controls the initial stiffness or linear elastic response of the material, while the parameter $\beta$ introduces and governs the degree of non-linearity, such as strain-limiting effects.  This adaptability allows the framework to represent a broad spectrum of mechanical responses, from classical linear elasticity to more complex behaviors. 

The function $\Psi_{1}(r)$ in \eqref{eq:Psi_specific} is monotone and ensures the resulting model can be hyperelastic or non-hyperelastic depending on the context \cite{mai2015strong,rajagopal2007response,rajagopal2011conspectus}. Crucially, it satisfies the limiting condition
\[
 \lim_{r \to \infty} r \Psi_{1}(r) \to \widehat{c}, \quad \mbox{where} \quad  \widehat{c} = 1/(2\mu \, \beta).
 \]
  A profound consequence of this model's constitutive limit becomes clear when analyzing the stress and strain fields near the tip of a crack with traction-free surfaces. For this case, a key physical insight emerges: while the \textit{stress norm} $\| \mathbf{T} \|$ becomes singular (unbounded) at the crack-tip, the corresponding \textit{strain norm} remains uniformly bounded. This behavior, governed by the relationship
\begin{equation*}
    \| \boldsymbol{\varepsilon} \| = \| \mathbf{T} \| \Psi_{1}(\| \mathbf{T} \|),
\end{equation*}
is physically significant because it prevents the non-physical prediction of infinite deformation at the point of failure, an artifact common to classical linear elastic models \cite{rajagopal2011modeling,mallikarjunaiah2015direct,kulvait2013,MalliPhD2015,ortiz2012numerical,bustamante2011solutions,ortiz2014numerical}. The limiting behavior of the function $\Psi_{1}$ ensures a finite strain response even in the presence of a stress singularity. This feature regularizes the solution at stress concentration areas. Also, the PDE \eqref{pde:mech} is a well-studied equation in mathematical analysis. It is similar to PDEs arising in models of nonlinear diffusion and is identical in form to the minimal surface equation from the calculus of variations. The existence and uniqueness of weak solutions for related BVPs have been established in works such as \cite{bulivcek2015existence,bulivcek2014elastic,bulivcek2015analysis}.

\section{Boundary value problem and existence of solution}
\label{sec:bvp_existence}

This section formally defines the BVP for the anti-plane shear model derived previously. We operate under the assumption of geometrical linearity, where material deformation is small, but the constitutive response is governed by the algebraically nonlinear relationship. The analysis considers a homogeneous, initially unstressed solid body occupying an infinitely long cylindrical domain $\mathcal{B} = \Omega \times \mathbb{R}$, where $\Omega \subset \mathbb{R}^2$ is a simply connected domain with a Lipschitz continuous ($C^{0,1}$) boundary, $\partial\Omega$. The boundary is partitioned into two non-overlapping parts, $\Gamma_D$ and $\Gamma_N$, where Dirichlet and Neumann conditions are applied, respectively.

\subsection{Function spaces and problem formulation}
We begin by defining the standard function spaces. Let $L^p(\Omega)$ for $p \in [1, \infty)$ be the space of Lebesgue integrable functions, and let $(\cdot, \cdot)_{L^2}$ and $\|\cdot\|_{L^2}$ denote the inner product and norm on the space of square-integrable functions, $L^2(\Omega)$. The classical Sobolev space $H^1(\Omega)$ is defined with the norm $\|v\|_{H^1(\Omega)}^2 := \|v\|_{L^2(\Omega)}^2 + \|\nabla v\|_{L^2(\Omega)}^2$. We also define the following essential spaces for our formulation:
\begin{equation}
    V := H^1(\Omega), \quad \text{and} \quad V_0 := \{ v \in H^1(\Omega) \mid v = 0 \text{ on } \Gamma_D \}.
\end{equation}
Here, $V_0$ is the space of test functions that vanish on the Dirichlet boundary.

As established in \cite{yoon2021quasi,yoon2022MAM,fernando2025}, the problem can be formulated entirely in terms of the Airy stress potential, $\Phi$. This elegantly reduces the original vector-valued traction and displacement problem into a single scalar, non-homogeneous Dirichlet problem. \\

\begin{cf}
Given $\mu$, $\beta$, and $\alpha$, find the stress potential $\Phi \in C^2(\overline{\Omega})$ such that:
\begin{subequations}\label{eq:strong_form}
\begin{align}
    -\nabla \cdot \left( \frac{\nabla \Phi}{2\mu \left(1 + \left( \beta  \|\nabla \Phi\| \right)^\alpha    \right)^{1/\alpha}} \right) &= 0 && \text{in } \Omega, \\
    \Phi &= \widehat{\Phi} && \text{on } \Gamma_D, \label{eq:dirichlet_bc} \\
    \left( \frac{\nabla \Phi}{2\mu \left(1 + \left( \beta  \|\nabla \Phi\| \right)^\alpha    \right)^{1/\alpha}} \right) \cdot \bfa{n} &= g && \text{on } \Gamma_N, \label{eq:neumann_bc}
\end{align}
\end{subequations}
where $\widehat{\Phi}$ and $g$ are prescribed functions and $\bfa{n}$ is the outward unit normal to $\Gamma_N$. For simplicity in this exposition, we consider a pure Dirichlet problem where $\partial\Omega = \Gamma_D$.
\end{cf}

To find a solution, the governing partial differential equation (PDE) is first recast into its weak formulation. 
This is accomplished by the standard procedure of multiplying the PDE by a suitable test function $\varphi \in V_0$ and integrating by parts over the domain.  This process transfers a derivative to the test function, thereby weakening the required smoothness of the solution $\Phi$. 
The resulting variational problem is: 
\begin{cwf}
Find $\Phi \in V$ that satisfies the essential boundary condition $\Phi = \widehat{\Phi}$ on the Dirichlet boundary $\Gamma_D$ such that for all test functions $\varphi \in V_0$:
\begin{equation}\label{eq:weak_form}
    \left( \frac{\nabla \Phi}{2\mu \left(1 + \left( \beta \|\nabla \Phi\| \right)^\alpha \right)^{1/\alpha}}, \nabla \varphi \right) = 0.
\end{equation}
\end{cwf}
This particular non-linear formulation, which describes a class of strain-limiting elastic solids, is an extension of similar frameworks investigated in \cite{bulivcek2014elastic,bulivcek2015existence,kulvait2019state,kulvait2013anti,yoon2022MAM}.

\subsection{Existence and Uniqueness of the Solution}
The well-posedness of the weak formulation \eqref{eq:weak_form} is non-trivial due to the nature of the nonlinearity. Fortunately, the governing equation is a specific case of a more general class of problems whose mathematical properties are well-established. The work of Bulíček et al. \cite{bulivcek2015existence} provides a rigorous proof of existence and uniqueness for such models. By adapting their main result to our specific case, we can assert the well-posedness of our problem.

\begin{Theorem}[Existence and Uniqueness, adapted from \cite{bulivcek2015existence}]
Let $\Omega \subset \mathbb{R}^2$ be a bounded domain with a $C^{0,1}$ boundary $\partial\Omega = \Gamma_D$. Assume there exists a function $\Phi^0 \in W^{1,\infty}(\Omega)$ such that its trace satisfies $\Phi^0|_{\Gamma_D} = \widehat{\Phi}$. Then the weak formulation \eqref{eq:weak_form} admits a unique solution $\Phi \in H^1(\Omega)$.
\end{Theorem}
\begin{proof}
The proof relies on the theory of monotone operators. The formulation \eqref{eq:weak_form} is a special instance of the problem analyzed in \cite{bulivcek2015existence} (with their model parameters set to $\alpha=1$ and $\beta=1$). The operator $\bfa{A}(\xi) := \xi / (1 + \|\xi\|)$ is strictly monotone and coercive, which are the key ingredients for applying standard existence theorems for quasilinear elliptic PDEs. For a detailed proof, the reader is referred to the original work in \cite{bulivcek2015existence}. 
\end{proof}

Since the governing nonlinear PDE \eqref{eq:strong_form} does not admit a closed-form analytical solution, even for simplified domains, a numerical approach is required. We therefore proceed to develop an $hp$-adaptive finite element method to approximate the solution to the weak formulation \eqref{eq:weak_form}. The subsequent section is dedicated to this task, where we introduce the discretization scheme. The proposed method employs adaptivity in $h$ and $p$, which is designed to enhance the accuracy of the numerical integration and the overall solution.

\section{Numerical Method} \label{sec:numerical_method}

The quasilinear nature of the governing PDE precludes the derivation of a closed-form analytical solution. Consequently, we employ a numerical approach based on the finite element method to obtain an approximate solution. A primary objective is to contrast the predictions of this nonlinear model with those of classical linear elasticity. While we use a standard $hp$-adaptive continuous Galerkin finite element method, a rigorous analysis of the discretization's optimality is beyond the scope of this paper and remains a topic for future investigation.

\subsection{Problem Statement and Linearization via Newton's Method}
Let $\Omega \subset \mathbb{R}^2$ be an open, bounded, and connected Lipschitz domain. Its boundary $\partial\Omega$ is composed of two disjoint parts, $\Gamma_D$ and $\Gamma_N$, where we assume $\Gamma_D \neq \emptyset$. We seek to solve the following BVP for the stress potential $\Phi$:
\begin{align} 
    -\nabla \cdot \left( \frac{\nabla \Phi}{2\mu(1 + \beta^{\alpha} \, \|\nabla \Phi\|^{\alpha})^{1/\alpha}} \right) &= 0 && \text{in } \Omega, \label{eq:pde_main} \\
    \Phi &= g(\bfa{x}) && \text{on } \Gamma_D, \label{eq:dirichlet_main}
\end{align}
where $g(\bfa{x})$ is a given continuous function. The derivation of appropriate boundary conditions for $\Phi$ from physical traction data is discussed in detail in \cite{kulvait2013}. To solve this nonlinear problem, we employ Newton's method. Let us define the nonlinear residual operator $F(\Phi)$ as:
\begin{equation}
    F(\Phi) := -\nabla \cdot \left( \frac{\nabla \Phi}{2\mu(1 + \beta^{\alpha} \, \|\nabla \Phi\|^{\alpha})^{1/\alpha}} \right).
\end{equation}
Given an iterate $\Phi^n$, the next approximation $\Phi^{n+1} = \Phi^n + \rho^n \;  \delta\Phi^n$ is found by solving a linear problem for the update $\delta\Phi^n$.  The term $\rho^n$ is a crucial damping factor, or step length, for the  $n$-th iteration of the solver. Its value is not fixed but is dynamically determined at each step by employing a line search algorithm. The primary goal of the line search is to find a suitable $\rho^n \in (0, 1]$ that ensures the update leads to a significant reduction in the nonlinear residual.  This strategy promotes global convergence and prevents the solver from taking overly aggressive steps that might lead to divergence, particularly when the  initial guess is far from the true solution. We employ the linearization at the PDE level for the  equation $F(\Phi^{n+1}) = 0$, which yields:
\begin{equation}\label{eq:newton_abstract}
    \nabla F(\Phi^n)[\delta\Phi^n] = -F(\Phi^n),
\end{equation}
where $\nabla F(\Phi^n)[\cdot]$ is the Gâteaux (or directional) derivative of $F$ at $\Phi^n$. This derivative can be computed as:
\begin{align}
    \nabla F(\Phi^n)[\delta\Phi^n] &:= \lim_{\varepsilon \to 0} \frac{F(\Phi^n + \varepsilon \delta\Phi^n) - F(\Phi^n)}{\varepsilon} \\
    &= -\nabla \cdot \left( \frac{\nabla \delta\Phi^n}{2\mu(1 + \beta\|\nabla \Phi^n\|^{\alpha})^{1/\alpha}} - \frac{\beta\|\nabla \Phi^n\|^{\alpha-2}(\nabla \Phi^n \cdot \nabla \delta\Phi^n)\nabla \Phi^n}{2\mu(1 + \beta\|\nabla \Phi^n\|^{\alpha})^{1/\alpha + 1}} \right).
\end{align}
Substituting this into \eqref{eq:newton_abstract}, we obtain the strong form of the linearized PDE for the update $\delta\Phi^n$.

\subsection{Variational formulation}
We now formulate the linearized problem in a weak form suitable for general finite element analysis. Multiplying \eqref{eq:newton_abstract} by a test function $\varphi \in H_0^1(\Omega)$ and integrating by parts, we obtain the linearized variational problem.

At each Newton step $n \ge 0$, find the update $\delta\Phi^n \in H_0^1(\Omega)$ such that for all $\varphi \in H_0^1(\Omega)$:
\begin{equation}\label{eq:weak_form_linearized}
\begin{split}
    & \left( \frac{\nabla \delta\Phi^n}{2\mu(1 + \beta^{\alpha} \|\nabla \Phi^n\|^{\alpha})^{1/\alpha}}, \nabla\varphi \right) - \left( \frac{\beta^{\alpha} \, \|\nabla \Phi^n\|^{\alpha-2}(\nabla \Phi^n \cdot \nabla \delta\Phi^n)\nabla \Phi^n}{2\mu(1 + \beta^{\alpha} \, \|\nabla \Phi^n\|^{\alpha})^{1/\alpha + 1}}, \nabla\varphi \right) \\
    & \qquad = -\left( \frac{\nabla \Phi^n}{2\mu(1 + \beta^{\alpha}\|\nabla \Phi^n\|^{\alpha})^{1/\alpha}}, \nabla\varphi \right).
\end{split}
\end{equation}

\begin{remark}
Newton's method requires an initial guess $\Phi^0$ that is sufficiently close to the true solution for convergence. A robust strategy, which we adopt, is to first solve the corresponding linear problem (by setting $\beta=0$) to obtain a good initial guess for the first iteration ($n=0$) of the nonlinear problem.
\end{remark}

\subsection{Finite element discretization}
Let $\{\mathcal{T}_h\}_{h>0}$ be a shape-regular family of partitions of the domain $\Omega$ into quadrilateral elements $\mathcal{K}$, with $h = \max_{\mathcal{K} \in \mathcal{T}_h} \text{diam}(\mathcal{K})$. We define the finite element space $V_h$ of continuous piecewise bilinear polynomials:
\[
    V_h := \{ v_h \in C^0(\overline{\Omega}) \mid v_h|_{\mathcal{K}} \in \mathbb{Q}_k(\mathcal{K}) \text{ for all } \mathcal{K} \in \mathcal{T}_h \} \subset H^1(\Omega).
\]
Here $\mathbb{Q}_k$ is the space of $k^{th}$ order polynomials. 
Let $V_{h,0}$ be the subspace of $V_h$ for functions vanishing on $\Gamma_D$. The discrete problem is to find the nodal values of the update $\delta\Phi_h^n \in V_{h,0}$ by solving the discretized version of \eqref{eq:weak_form_linearized}:
\begin{equation}\label{eq:discrete_system}
    a_{\Phi_h^n}(\delta\Phi_h^n, \varphi_h) = l_{\Phi_h^n}(\varphi_h) \quad \forall \varphi_h \in V_{h,0},
\end{equation}
where the iterate-dependent bilinear form $a_{\Phi_h^n}(\cdot, \cdot)$ and linear functional $l_{\Phi_h^n}(\cdot)$ are defined as:
\begin{align}
    a_{\Phi_h^n}(w_h, \varphi_h) &:= \left( \frac{\nabla w_h}{2\mu(1 + \beta^{\alpha}\|\nabla \Phi_h^n\|^{\alpha})^{1/\alpha}}, \nabla\varphi_h \right) \nonumber \\
    & \qquad - \left( \frac{\beta^{\alpha}\|\nabla \Phi_h^n\|^{\alpha-2}(\nabla \Phi_h^n \cdot \nabla w_h)\nabla \Phi_h^n}{2\mu(1 + \beta^{\alpha}\|\nabla \Phi_h^n\|^{\alpha})^{1/\alpha + 1}}, \nabla\varphi_h \right), \label{eq:bilinear_form_a} \\
    l_{\Phi_h^n}(\varphi_h) &:= -\left( \frac{\nabla \Phi_h^n}{2\mu(1 + \beta^{\alpha}\|\nabla \Phi_h^n\|^{\alpha})^{1/\alpha}}, \nabla\varphi_h \right). \label{eq:linear_functional_l}
\end{align}
At each step, we solve the linear system of equations resulting from \eqref{eq:discrete_system} to find the discrete update $\delta\Phi_h^n$, and then update the solution via $\Phi_h^{n+1} = \Phi_h^n + \rho^n \delta\Phi_h^n$ until the residual, given by the norm of $l_{\Phi_h^n}$, falls below a prescribed tolerance. 

\subsection{An exposition of $hp$-adaptive finite element methods}

The pursuit of accuracy and efficiency in finite element analysis has led to the development of adaptive methods, which dynamically adjust the discretization to the specific characteristics of the problem being solved. These methods automatically refine the simulation in regions where the numerical error is high, thereby optimizing computational effort. Two fundamental approaches to adaptivity exist: $h$-adaptivity and $p$-adaptivity. The $hp$-adaptive strategy represents a robust synthesis of these two paradigms, aiming to achieve superior convergence rates by making an intelligent, localized choice between them. To understand the $hp$-adaptive method, one must first appreciate its constituent parts. Each strategy offers a distinct mechanism for reducing discretization error.

\subsubsection{$h$-Adaptivity: Refining the mesh}
The more traditional of the two, $h$-adaptivity, involves reducing the size, $h$, of the elements in the computational mesh. The core principle is that a finer mesh provides a better resolution of the solution \cite{bangerth2013adaptive,zienkiewicz1991adaptivity}. This method is particularly effective at handling solutions with sharp, localized features such as boundary layers, steep gradients, or singularities (e.g., at re-entrant corners) \cite{gresho1981,verfurth1994posteriori,verfurth1994posteriori}. By selectively subdividing cells in these problematic areas, $h$-refinement can accurately capture geometric complexities and non-smooth solution behavior. However, for problems where the solution is globally smooth, $h$-adaptivity can be suboptimal. It may require an excessively large number of small elements to achieve a desired accuracy, leading to a dramatic increase in the number of degrees of freedom and, consequently, high computational cost.

The core idea of any residual-based error estimator is that the true error, $e = u - u_h$, where $u$ is the exact solution and $u_h$ is the finite element solution, is driven by the extent to which $u_h$ fails to satisfy the original partial differential equation \cite{de1983posteriori}. This failure is quantified by the \textit{residual}. For a simple Poisson problem, $-\Delta u = f$, the error is mathematically linked to two forms of the residual:
\begin{enumerate}
    \item \textbf{The Element Residual ($R_K$):} This measures how well the equation is satisfied \textit{inside} each element $K$. It is defined as $R_K = f + \Delta u_h$.
    \item \textbf{The Jump Residual ($J_F$):} This arises because, for standard continuous finite elements, the gradient of the numerical solution, $\nabla u_h$, is typically discontinuous across the interior faces of the mesh. The jump residual measures the magnitude of this discontinuity in the normal flux across each face $F$.
\end{enumerate}

The \textit{Kelly error estimator} \cite{kelly1983posteriori,de1983posteriori} combines these residuals into a single, computable value, $\eta_K$, for each cell $K$ in the mesh. This value serves as an indicator of the local error. The formula for the estimator on a given cell $K$ can be expressed as:
$$
\eta_K^2 = h_K^2 \int_K (f + \Delta u_h)^2 \,dx + \sum_{F \in \partial K} h_F \int_F \left( \jump{\nabla u_h \cdot \mathbf{n}} \right)^2 \,ds
$$
Let's break down the components of this formula:
\begin{itemize}
    \item \textbf{$\eta_K$}: This is the error indicator for cell $K$. While not a strict upper bound on the true error, it is proven to be proportional to it, making it an excellent guide for refinement.
    \item \textbf{$h_K$ and $h_F$}: These are the diameters of the cell $K$ and face $F$, respectively. These scaling factors are crucial for ensuring the estimator has the correct physical units and for balancing the contributions of the element and jump residuals.
    \item \textbf{The element residual term}: The first term, involving the integral over the cell $K$, accounts for the error originating from within the element. For linear finite elements ($p=1$), the term $\Delta u_h$ is zero, and this integral simplifies. In many practical implementations, this term is often neglected. This is because the jump residual is computationally cheaper to evaluate and is frequently the dominant source of error, especially in regions where the solution's second derivatives are large.
    \item \textbf{The Jump Residual Term}: The second term is the heart of the Kelly estimator. The notation $\jump{\nabla u_h \cdot \mathbf{n}}$ represents the jump in the normal component of the gradient across each face $F$ of the cell $K$.
    \begin{itemize}
        \item For an \textbf{interior face} $F$ shared by two cells, $K_1$ and $K_2$, the jump is the difference in the normal derivatives as computed from each side: $\jump{\nabla u_h \cdot \mathbf{n}} = \nabla u_h|_{K_1} \cdot \mathbf{n}_1 + \nabla u_h|_{K_2} \cdot \mathbf{n}_2$.
        \item For a \textbf{boundary face} $F$ lying on the domain boundary $\partial\Omega$, the jump measures the difference between the computed flux and the flux prescribed by Neumann boundary conditions ($g = \nabla u \cdot \mathbf{n}$). The term becomes $\nabla u_h \cdot \mathbf{n} - g$. If the boundary condition is of Dirichlet type, this term is zero.
    \end{itemize}
\end{itemize}
The intuition is powerful: a large jump in the gradient signifies a ``kink'' in the numerical solution's derivative. This indicates that the true solution likely has high curvature (large second derivatives) in that area, which the current piecewise polynomial approximation is struggling to capture. The estimator correctly identifies these regions as having high error.

In static fracture analysis, the solution's gradient is singular at the crack-tip. A numerical approximation on a coarse mesh will invariably fail to capture this singularity, resulting in significant, non-physical flux discontinuities between the elements surrounding the crack-tip. The normal jump term therefore acts as a natural and robust indicator, isolating the region of high error. Consequently, when used to drive an adaptive mesh refinement strategy, it ensures that computational resources are concentrated precisely where they are most needed—at the crack-tip—enabling an accurate and efficient resolution of the singular crack-tip fields.

The vector of cell-wise error indicators, $\{\eta_K\}$, provided by the Kelly estimator is the essential input for an $h$-adaptive strategy. The goal is to refine the mesh where the error is significant and, optionally, coarsen it where the error is small. This is implemented in our code through a simple yet effective loop: 
\[
\textbf{SOLVE $\rightarrow$ ESTIMATE $\rightarrow$ MARK $\rightarrow$ REFINE}.
\]
\begin{enumerate}
    \item \textbf{SOLVE}: The finite element problem is solved on the current mesh to obtain the solution $u_h$.
    \item \textbf{ESTIMATE}: This step takes the finite element solution vector and other relevant data (like boundary conditions) and computes the vector of indicators, \texttt{eta}, where \texttt{eta[i]} holds the value of $\eta_K$ for the \texttt{i}-th cell.
    \item \textbf{MARK}: This is the crucial decision-making step where cells are flagged for refinement or coarsening based on the estimated error. The core ideas behind this function are as follows:
    \begin{itemize}
        \item It sorts the cells based on their error indicator values, from largest to smallest.
        \item It marks a certain fraction of cells with the highest error for refinement.
        \item It marks a different fraction of cells with the lowest error for coarsening.
    \end{itemize}
    \item \textbf{REFINE}: The triangulation is then modified according to the flags set in the marking step. Our code automatically handles the complex bookkeeping involved in this process, including the creation of \textit{hanging nodes} where a refined cell meets an unrefined one. It then sets up the necessary constraints to ensure the finite element solution remains continuous across these non-standard interfaces.  
\end{enumerate}
This adaptive cycle is repeated until a desired level of accuracy or a maximum number of cells is reached. The result is a highly optimized mesh, automatically tailored to the problem's unique features, which delivers a more accurate solution for a given amount of computational work.

\subsubsection{$p$-Adaptivity: Enriching the approximation space}

In contrast, $p$-adaptivity leaves the mesh geometry unchanged and instead increases the polynomial degree, $p$, of the basis functions used on the elements \cite{babuska1981p,duster2001p}. By moving from linear to quadratic, cubic, or even higher-order polynomials, the approximation space becomes richer and can represent more complex functions on the same element. This approach is exceptionally powerful for problems where the solution is smooth. For such solutions, $p$-adaptivity can yield exponential rates of convergence, a phenomenon often referred to as spectral convergence. This means that a desired accuracy can be reached with far fewer degrees of freedom compared to $h$-refinement. Its primary weakness, however, is its relative ineffectiveness in the presence of singularities or non-smooth features, where the theoretical advantages of high-order polynomials are diminished.

\subsubsection{The \textit{hp}-Strategy: A synergistic hybrid approach}

The $hp$-adaptive method is developed from the recognition that neither $h$- nor $p$-adaptivity is universally superior. The optimal strategy depends on the local character of the solution. The $hp$ method, therefore, does not commit to one approach but instead makes a dynamic, cell-by-cell decision to either subdivide an element ($h$-refinement) or increase its polynomial degree ($p$-refinement) \cite{bonito2024adaptive}. The central goal is to leverage the strengths of each method precisely where it is most effective:
\begin{itemize}
    \item Use $h$-refinement to resolve geometric singularities and sharp, non-smooth fronts.
    \item Use $p$-refinement to capture smooth solution behavior efficiently.
\end{itemize}
By doing so, the $hp$-method can theoretically achieve exponential convergence rates even for the broad class of problems that involve both smooth regions and localized singularities.

\subsubsection{The \textit{decision}-\textit{marking} criterion}

The effectiveness of the $hp$-adaptive method is critically dependent on its decision-making algorithm, which arbitrates between local mesh refinement ($h$) and polynomial degree enrichment ($p$). The specific strategy implemented in this work provides a robust foundation for future investigations and extensions. The process can be broken down into two main stages for each cell flagged for refinement by a global error estimator. First, an a posteriori error estimator, such as the Kelly error indicator (described in the previous subsection), is used to approximate the numerical error on every cell of the mesh. Cells whose estimated error contributes significantly to the total global error are marked as candidates for refinement. This step identifies \emph{where} the discretization is insufficient but does not yet determine \emph{how} to improve it.

Once a cell is marked for refinement, the $hp$-algorithm must decide between subdividing it or enriching its polynomial degree. This choice is based on an analysis of the solution's local smoothness within that cell. A clever way to probe this smoothness is to compare the existing numerical solution on the cell with a projection of that solution onto a lower-order polynomial space. The intuition is as follows:
\begin{itemize}
    \item If the solution on a cell is very smooth (analytic), it can be well-approximated by polynomials. The difference between the solution and its projection onto a lower-degree space will be relatively large, as the higher-order components are significant. This indicates that the solution has complex features that a higher polynomial degree could capture more efficiently. In this scenario, the algorithm chooses \textbf{$p$-refinement}.

    \item Conversely, if the solution is not smooth (e.g., near a singularity), its energy is spread more evenly across all polynomial modes. The projection onto a lower-degree space will capture a substantial part of the solution, and the difference will be relatively small compared to the total error on the cell. This suggests that simply increasing the polynomial degree will not be very effective. The error is more likely due to a feature that requires finer geometric resolution. In this case, the algorithm opts for \textbf{$h$-refinement}.
\end{itemize}
This elegant logic allows the simulation to tailor its discretization strategy to the local physics of the problem, creating highly optimized meshes that deliver exceptional accuracy for a given computational budget. The implementation of such a scheme requires specialized data structures. Our code, which is based on the open-source \textsf{deal.II} library is capable of managing a mesh where different cells can have different polynomial degrees.

\section{Results and discussions}\label{rd}
This section is dedicated to the numerical verification of the proposed strain-limiting fracture model, with a focus on its capacity to produce bounded strains in the presence of unbounded, near-tip stresses. The governing partial differential equation \eqref{pde:mech} is discretized using an advanced $hp$-adaptive finite element method. This strategy is particularly well-suited for resolving the singular fields characteristic of fracture problems, as it synergistically combines local mesh refinement ($h$-refinement) with local increases in the polynomial degree of the basis functions ($p$-refinement). Our implementation, developed within the robust, open-source \texttt{deal.II} finite element library \cite{arndt2021deal}, leverages a sophisticated decision-making algorithm to determine the optimal refinement strategy for each region of the domain, ensuring an efficient allocation of computational resources.

To address the inherent nonlinearity of the second-order quasilinear partial differential equation, we employ a Newton-Raphson scheme. This iterative method approximates the nonlinear problem with a sequence of linear ones by computing the Fréchet derivative of the underlying operator at the current solution estimate. Consequently, each iteration necessitates the numerical solution of a large, sparse linear system of algebraic equations for the incremental update. The process is deemed to have converged and is subsequently terminated when the Euclidean norm of the residual vector falls below a prescribed tolerance.

To demonstrate the framework's efficacy and versatility, we investigate a scenario involving a classical anti-plane shear crack problem, which serves as a benchmark for $ hp$-adaptive finite element implementation and the analysis of the crack-tip field predicted by the nonlinear model.  A primary objective of this numerical investigation is to elucidate the regularizing effect of the strain-limiting constitutive law. Specifically, we demonstrate that the model yields bounded strains in the vicinity of stress concentrations, a result that stands in stark contrast to the non-physical, square-root singular strain fields predicted by classical linear elastic fracture mechanics.  While the present work utilizes a continuous Galerkin formulation, the robust performance of the $hp$-adaptive framework underscores its suitability for these challenging problems. A rigorous a priori error analysis of the discretization, similar to the one done in \cite{manohar2024hp}, is a subject of ongoing research. 

The following algorithm illustrates all the steps involved in obtaining the numerical solution based on the $ hp$-adaptive finite element method.

\begin{algorithm}[H]
    \caption{$hp$-Adaptive Finite Element Method for Quasilinear PDE}
    \label{alg:hp_adaptive_fem}
    \begin{algorithmic}[1]
        \State \textbf{Initialization:}
        \State Set adaptive loop counter $k=0$.
        \State Define initial coarse mesh $\mathcal{T}_0$ and initial polynomial degree distribution $\mathbb{Q}_0$.
        \State Initialize solution $\Phi_{h,k}^0$ (e.g., from a trivial state or previous computation).
        \State Define tolerances: $\text{TOL}_{\text{Newton}}$, $\text{TOL}_{\text{adapt}}$.
        \State Define marking parameters: $0 < \theta_p, \theta_h < 1$ with $\theta_h + \theta_p < 1$.

        \Repeat \Comment{Main Adaptive Loop}
            \State Set Newton iteration counter $n=0$.
            \State Set current solution iterate $\Phi_{h,k}^{n} \leftarrow \Phi_{h,k}$.

            \Repeat \Comment{Damped Newton's Method Loop}
                \State Assemble the Jacobian matrix and residual vector based on the current iterate $\Phi_{h,k}^{n}$:
                \State \qquad $a_{\Phi_{h,k}^n}(\cdot, \cdot)$ from Eq. \eqref{eq:bilinear_form_a}
                \State \qquad $l_{\Phi_{h,k}^n}(\cdot)$ from Eq. \eqref{eq:linear_functional_l}
                \State Solve the linear system for the update: Find $\delta\Phi_{h,k}^n \in V_{h,0}$ such that
                \State \qquad $a_{\Phi_{h,k}^n}(\delta\Phi_{h,k}^n, \varphi_h) = l_{\Phi_{h,k}^n}(\varphi_h) \quad \forall \varphi_h \in V_{h,0}$.
                
                \State \textbf{Line Search:} Find a damping parameter $\rho^n \in (0, 1]$ that ensures a sufficient decrease in the residual norm, i.e., find $\rho^n$ such that:
                \State \qquad $\norm{l_{\Phi_{h,k}^n + \rho^n \delta\Phi_{h,k}^n}} < (1 - \gamma \rho^n) \norm{l_{\Phi_{h,k}^n}}$ for some $\gamma \in (0, 1)$.
                
                \State Update the solution: $\Phi_{h,k}^{n+1} \leftarrow \Phi_{h,k}^n + \rho^n \delta\Phi_{h,k}^n$.
                \State Increment Newton counter: $n \leftarrow n+1$.
            \Until{$\norm{l_{\Phi_{h,k}^n}} < \text{TOL}_{\text{Newton}}$} \Comment{Stopping criterion for Newton's method}

            \State Set converged solution for this adaptive step: $\Phi_{h,k} \leftarrow \Phi_{h,k}^{n}$.

            \State \textbf{A Posteriori Error Estimation:}
            \State For each element $\tau \in \mathcal{T}_k$, compute a local error indicator $\eta_{\tau}(\Phi_{h,k})$.
            \State Compute the global error estimate $\eta_k = \left( \sum_{\tau \in \mathcal{T}_k} \eta_{\tau}^2 \right)^{1/2}$.

            \If{$\eta_k < \text{TOL}_{\text{adapt}}$}
                \State \textbf{break} \Comment{Converged, exit adaptive loop}
            \EndIf

            \State \textbf{Marking Strategy:}
            \State Mark a subset of elements $\mathcal{M}_h \subset \mathcal{T}_k$ for $h$-refinement, typically those with the largest error indicators, such that $\sum_{\tau \in \mathcal{M}_h} \eta_{\tau}^2 \geq \theta_h^2 \eta_k^2$.
            \State Mark a subset of elements $\mathcal{M}_p \subset \mathcal{T}_k \setminus \mathcal{M}_h$ for $p$-enrichment, typically those with high solution regularity, such that $\sum_{\tau \in \mathcal{M}_p} \eta_{\tau}^2 \geq \theta_p^2 \eta_k^2$.

            \State \textbf{Adaptivity:}
            \State Generate a new mesh $\mathcal{T}_{k+1}$ by refining all elements $\tau \in \mathcal{M}_h$.
            \State Define a new polynomial degree distribution $\mathbb{Q}_{k+1}$ by increasing the degree for all elements $\tau \in \mathcal{M}_{\mathbb{Q}}$.
            \State Increment adaptive loop counter: $k \leftarrow k+1$.
            \State Project the solution $\Phi_{h,k-1}$ onto the new finite element space $(\mathcal{T}_k, \mathbb{Q}_k)$ to get an initial guess $\Phi_{h,k}$.

        \Until{a maximum number of adaptive steps is reached or convergence is achieved.}
        
        \State \textbf{return} Converged solution $\Phi_{h,k}$.
    \end{algorithmic}
\end{algorithm}

\subsection{$h$-convergence study}
Our numerical solver, specifically the implementation employed for the case without $hp$-adaptivity, is directly adopted from and precisely the same as the methodology detailed in \cite{yoon2022MAM}. To validate our implementation and ensure consistency with prior work, we conducted an identical $h$-convergence study as reported in \cite{yoon2022MAM}. The results of this study yielded the same convergence rates, thereby confirming the correctness and reliability of our solver. Consequently, to maintain conciseness and focus on the novel contributions of the present work, we refrain from reiterating these established details here. For a comprehensive exposition of the solver's intricacies and the specifics of the $h$-convergence analysis, interested readers are kindly directed to the aforementioned article \cite{yoon2022MAM}.

\subsection{$hp$-adaptive FEM simulation of static crack problem}
The problem of static anti-plane shear crack modeled using algebraically nonlinear constitutive relationships reduces ot solving a second-order quasilinear partial differential equation. The algorithm~\ref{alg:hp_adaptive_fem} is adopted for obtaining the numerical solution to the discrete weak formulation defined in Equations \eqref{eq:bilinear_form_a} and \eqref{eq:linear_functional_l}. We consider a unit square with a crack lying  $\;\; 0.5 \leq x \leq 1, \;\; y=0.5$. The domain and boundary conditions are shown below in Figure~\ref{fig_domain}.

\begin{figure}[H]
\centering
\begin{tikzpicture}[
    scale=5, 
    domain_boundary/.style={draw=black, thick},
    crack_active/.style={line width=0.8mm, blue!80, cap=round}, 
    crack_potential/.style={dashed, line width=0.6mm, gray!60, cap=round}, 
    label_coords/.style={font=\small, text=black}, 
    label_bc/.style={font=\small, text=black},     
    label_crack/.style={font=\small, text=black}    
]

\fill[gray!8] (0,0) rectangle (1,1);
\draw[domain_boundary] (0,0) rectangle (1,1);
\draw[crack_potential] (0,0.5) -- (0.5,0.5);
\draw[crack_active] (0.5,0.5) -- (1,0.5);
\node[label_coords, anchor=north east] at (0,0) {$(0,0)$};
\node[label_coords, anchor=north west] at (1,0) {$(1,0)$};
\node[label_coords, anchor=south west] at (1,1) {$(1,1)$};
\node[label_coords, anchor=south east] at (0,1) {$(0,1)$};
\node[label_bc, rotate=90, anchor=south] at (0,0.75) {$\Gamma^{*}_L: \Phi(0,y)=1$};
\node[label_bc, rotate=90, anchor=north] at (1,0.75) {$\Gamma_R: \Phi(1,y)=0$};
\node[label_bc, anchor=north] at (0.5,0) {$\Gamma^{*}_{B}: \Phi(x,0)=1-x$};
\node[label_bc, anchor=south] at (0.5,1) {$\Gamma^{*}_{T}: \Phi(x,1)=1-x$};
\node[label_crack, above=2pt, xshift=0.1cm] at (0.75,0.5) {$\Gamma^{*}_{C}$}; 
\node[label_crack, below=2pt, xshift=0.1cm] at (0.75,0.5) {$\Phi(x,0.5)=0$}; 

\end{tikzpicture}
\caption{A square domain containing a single edge crack}
\label{fig_domain}
\end{figure}
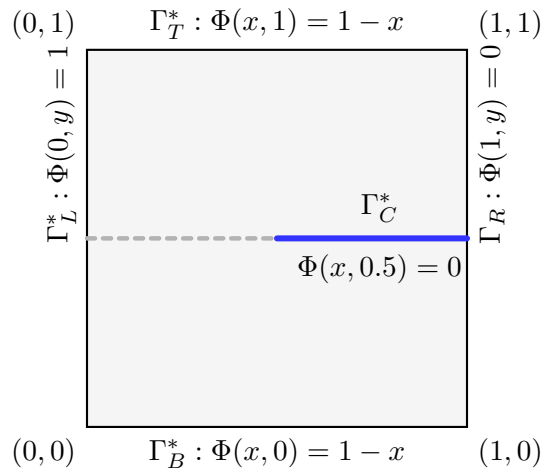

For this specific numerical example, the detailed computational setup was configured as follows. The initial coarse mesh resolution was set to $h=1/8$, serving as the starting point for the adaptive refinement process. A comprehensive parameter study was conducted by systematically varying the key model parameters $\alpha$ and $\beta$ across the discrete set of values $\{0.5, 1.0, 2.0, 5.0, 10.0\}$ for each. Concurrently, the material shear modulus, a critical mechanical property, was held constant at a value of $1.0$ throughout all simulations. To achieve high accuracy and efficiently capture the complex features of the solution, the numerical computations leveraged a combined $hp$-adaptive strategy. This involved performing up to eight levels of adaptive mesh refinement ($h$-refinement), allowing the mesh to become progressively finer in regions requiring higher resolution. Simultaneously, the polynomial degree ($p$) of the finite element basis functions was adaptively increased, reaching up to the $7^{th}$-order in elements where higher approximation capabilities were beneficial. The entire $hp$-adaptive solution process, integrating both mesh refinement and polynomial enrichment, was orchestrated and executed using the robust framework outlined in Algorithm~\ref{alg:hp_adaptive_fem} to efficiently obtain the numerical solution.

\begin{figure}[H]
    \centering 

    \begin{subfigure}[b]{0.48\textwidth}
        \includegraphics[width=\linewidth]{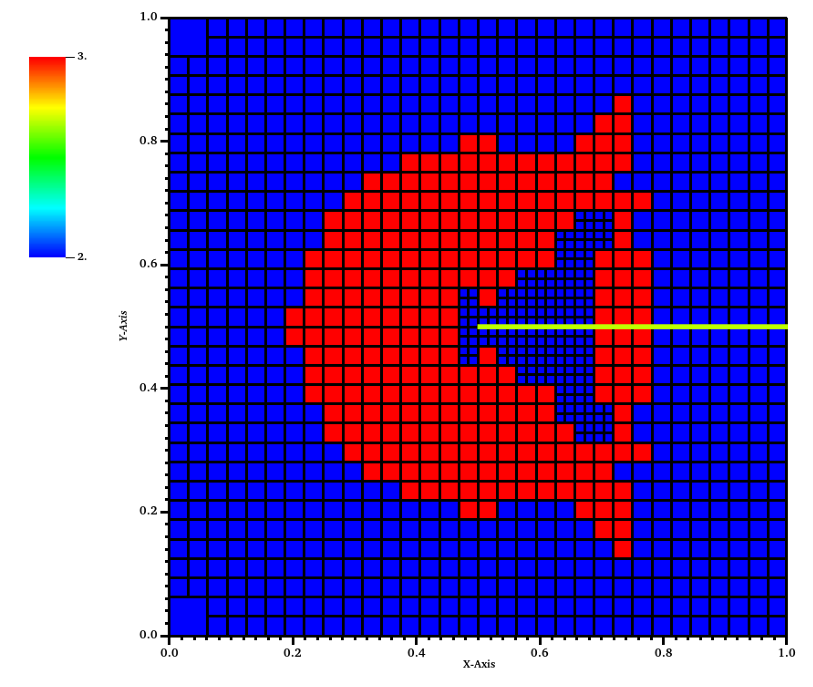}
        \caption{ $hp$-refinement-1}
        \label{fig:fig1}
    \end{subfigure}
    \hfill 
    \begin{subfigure}[b]{0.48\textwidth}
        \includegraphics[width=\linewidth]{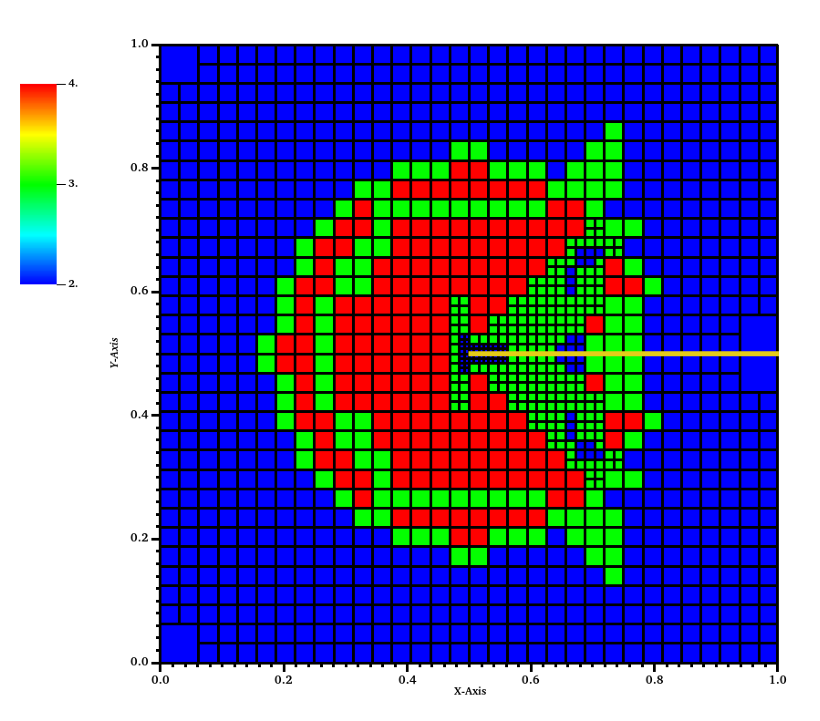}
        \caption{$hp$-refinement-2}
        \label{fig:fig2}
    \end{subfigure}

    \vspace{0.5cm} 

    \begin{subfigure}[b]{0.48\textwidth}
        \includegraphics[width=\linewidth]{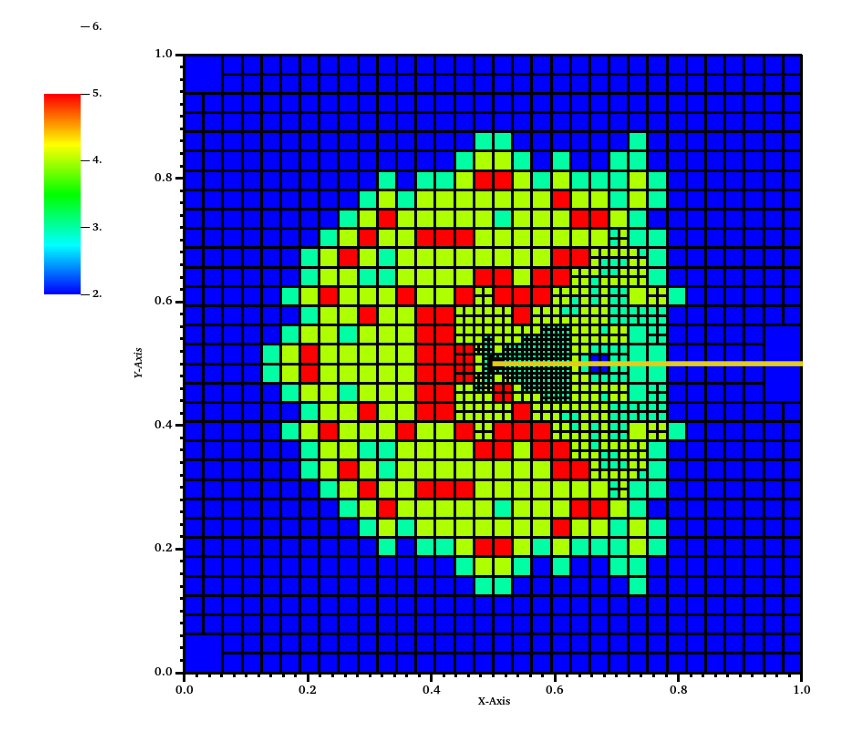}
        \caption{$hp$-refinement-3}
        \label{fig:fig3}
    \end{subfigure}
    \hfill 
    \begin{subfigure}[b]{0.48\textwidth}
        \includegraphics[width=\linewidth]{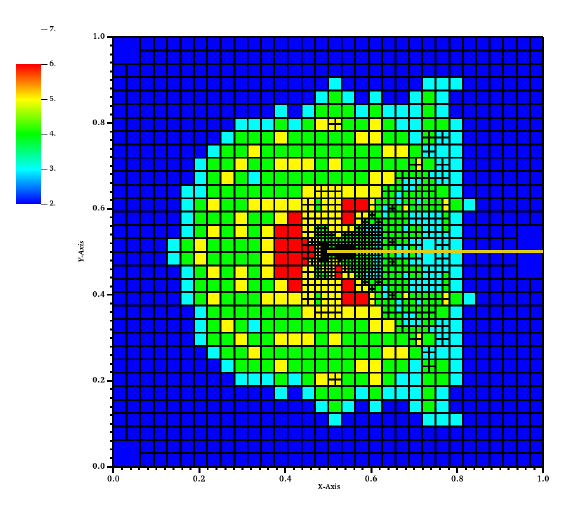}
        \caption{$hp$-refinement-4}
        \label{fig:fig4}
    \end{subfigure}

    \vspace{0.5cm} 

    \begin{subfigure}[b]{0.48\textwidth}
        \includegraphics[width=\linewidth]{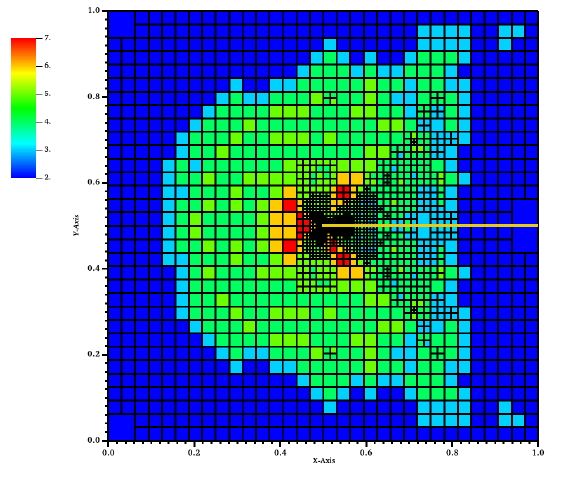}
        \caption{$hp$-refinement-5}
        \label{fig:fig5}
    \end{subfigure}
    \hfill 
    \begin{subfigure}[b]{0.48\textwidth}
        \includegraphics[width=\linewidth]{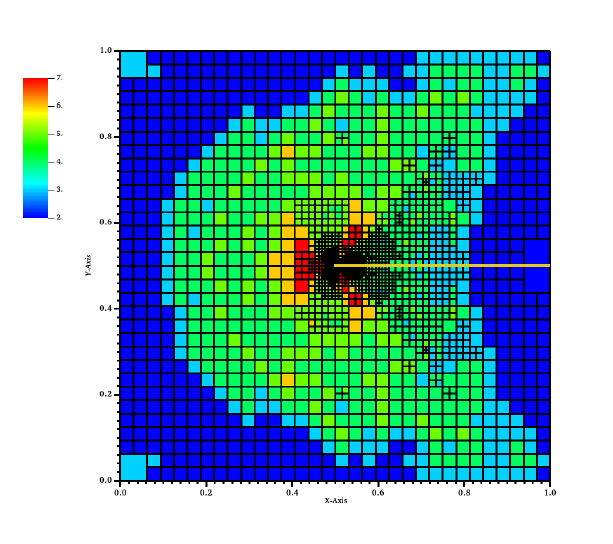}
        \caption{$hp$-refinement-6}
        \label{fig:fig6}
    \end{subfigure}

    \caption{The six $hp$-finements}
    \label{fig:six_figures_total}
\end{figure}

\begin{figure}[H]
    \centering 

    \begin{subfigure}[b]{0.48\textwidth}
        \includegraphics[width=\linewidth]{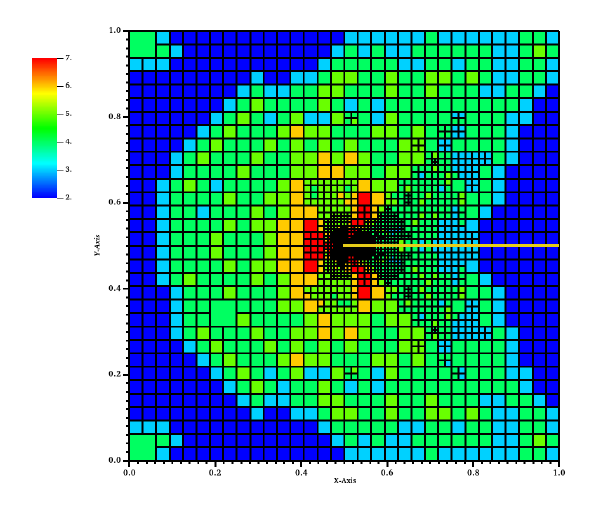}
        \caption{$hp$-refinement-7}
        \label{fig:fig1}
    \end{subfigure}
    \hfill 
    \begin{subfigure}[b]{0.48\textwidth}
        \includegraphics[width=\linewidth]{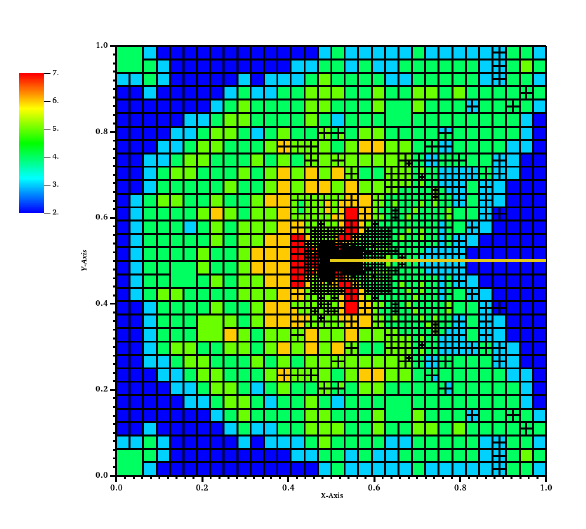}
        \caption{$hp$-refinement-8}
        \label{fig:fig2}
    \end{subfigure}

    \vspace{0.5cm} 

    \begin{subfigure}[b]{0.48\textwidth}
        \includegraphics[width=\linewidth]{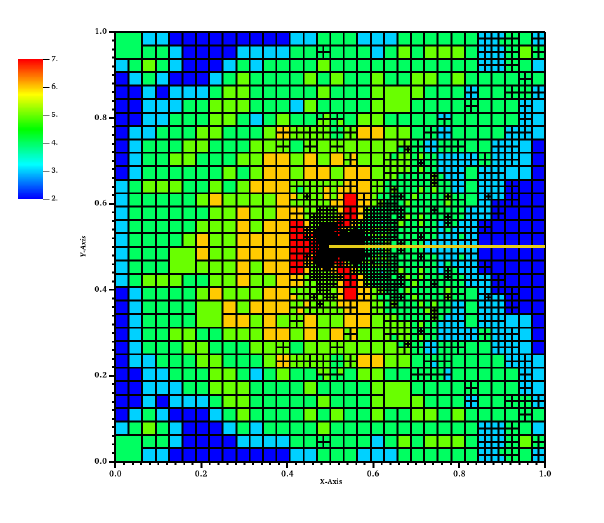}
        \caption{$hp$-refinement-9}
        \label{fig:fig3}
    \end{subfigure}
    \hfill 
    \begin{subfigure}[b]{0.48\textwidth}
        \includegraphics[width=\linewidth]{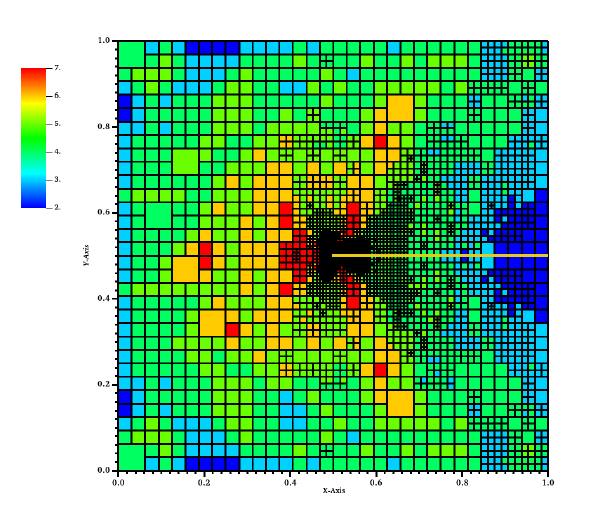}
        \caption{$hp$-refinement-10}
        \label{fig:fig4}
    \end{subfigure}

    \caption{The last 4 $hp$-refinements .}
    \label{fig:four_figures}
\end{figure}

The interplay of $h$-refinement and $p$-refinement across the first ten refinement stages is dynamically represented in Figures \ref{fig:six_figures_total} and \ref{fig:four_figures}. While h-refinement is notably extensive in the immediate vicinity of the crack-tip, p-refinement simultaneously demonstrates a clear and significant presence around the stress concentrator. This highlights how both refinement strategies contribute to accurately resolving the solution in different critical areas. Each distinct color within these figures serves to visually represent the polynomial order employed in approximating the solution to the BVP. This allows for a quick and intuitive understanding of how the approximation's complexity varies across different regions.

\begin{figure}[H]
    \centering 
    \begin{subfigure}[b]{0.48\textwidth}
        \includegraphics[width=\linewidth]{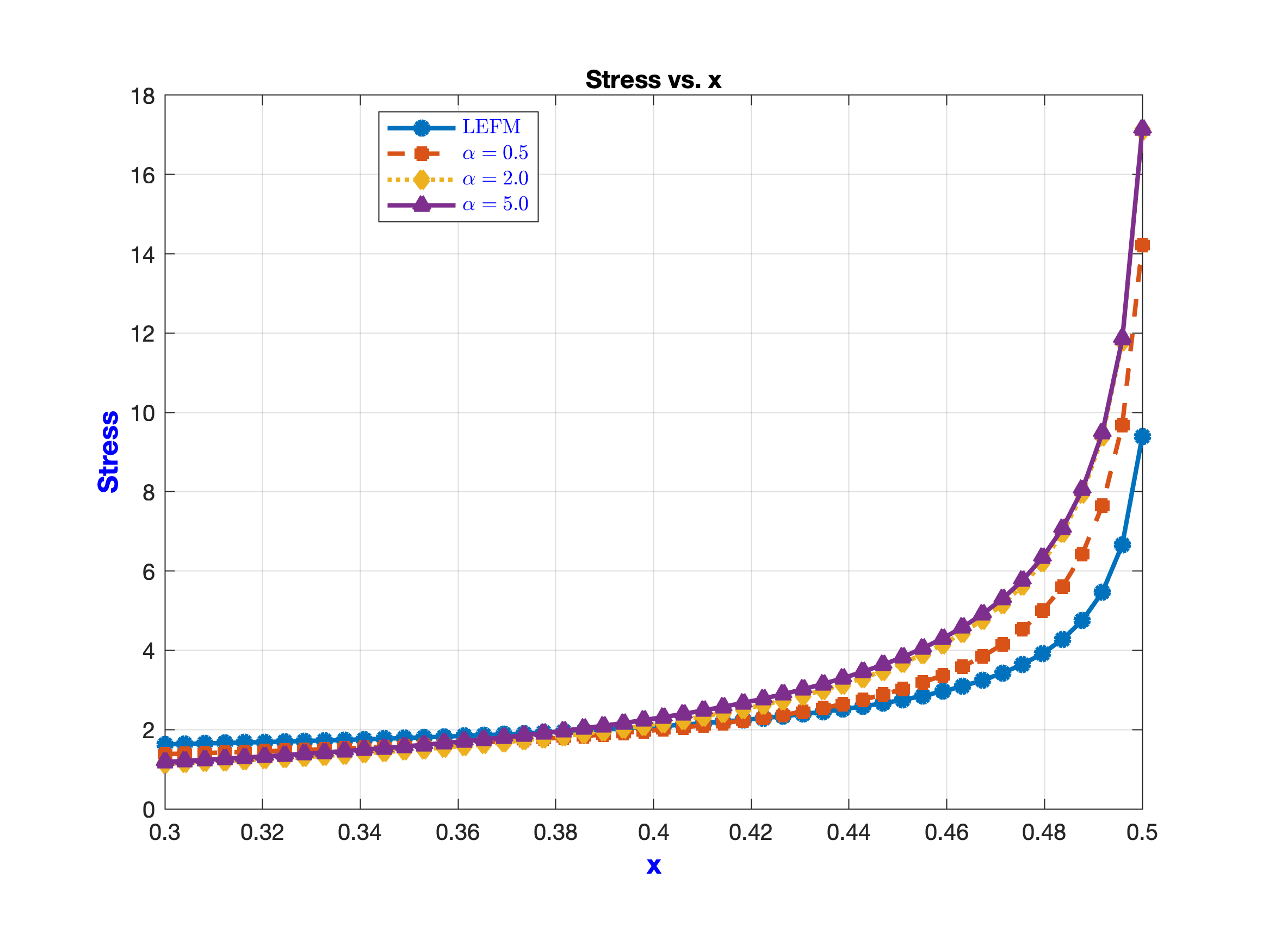}
        \caption{Effect of $\alpha$ on $\bfa{T}_{23}$ }
        \label{fig:fig1}
    \end{subfigure}
    \hfill 
    \begin{subfigure}[b]{0.48\textwidth}
        \includegraphics[width=\linewidth]{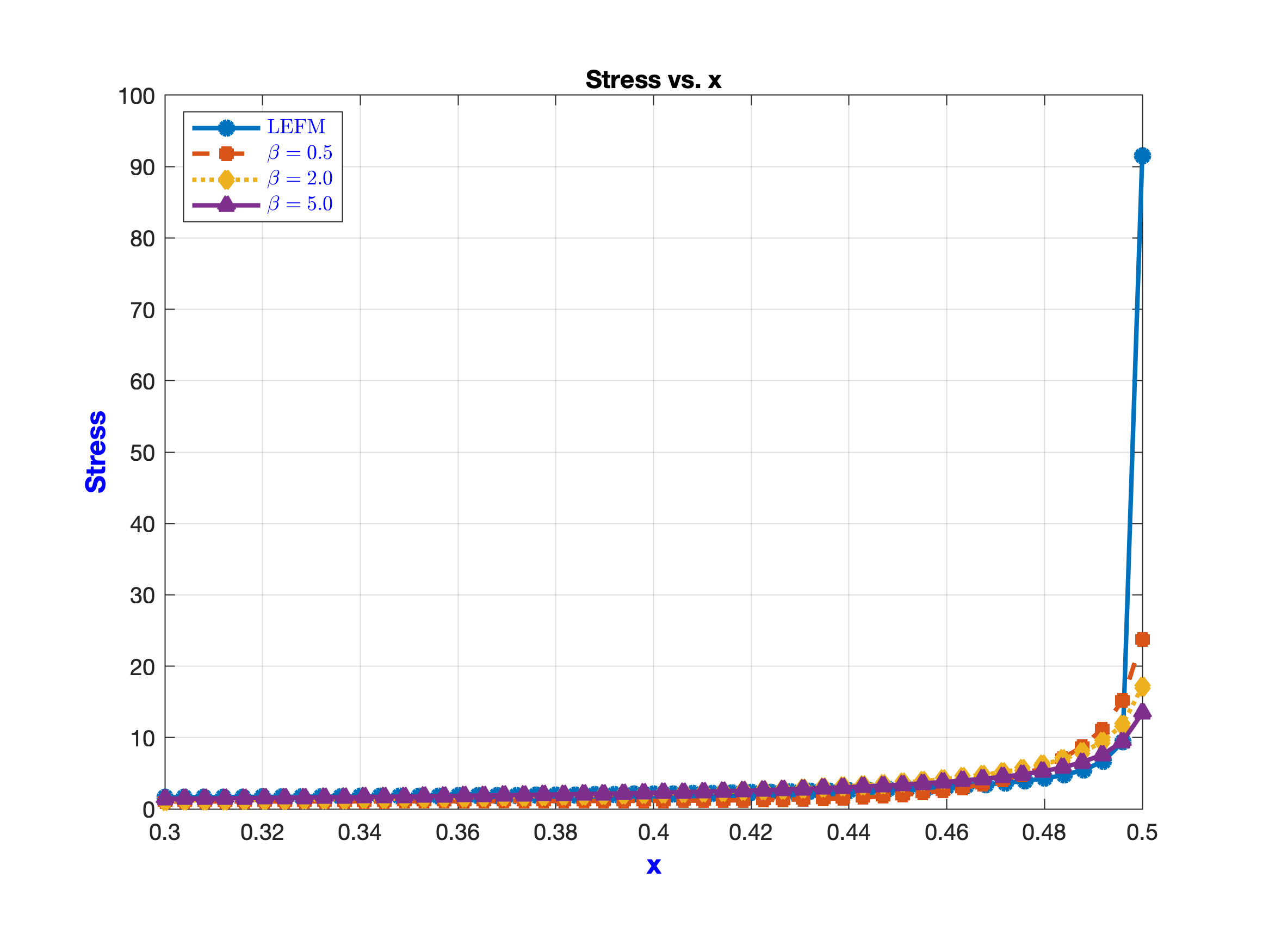}
        \caption{Effect of $\beta$ on $\bfa{T}_{23}$}
        \label{fig:fig2}
    \end{subfigure}
    \vspace{0.5cm} 
    \begin{subfigure}[b]{0.48\textwidth}
        \includegraphics[width=\linewidth]{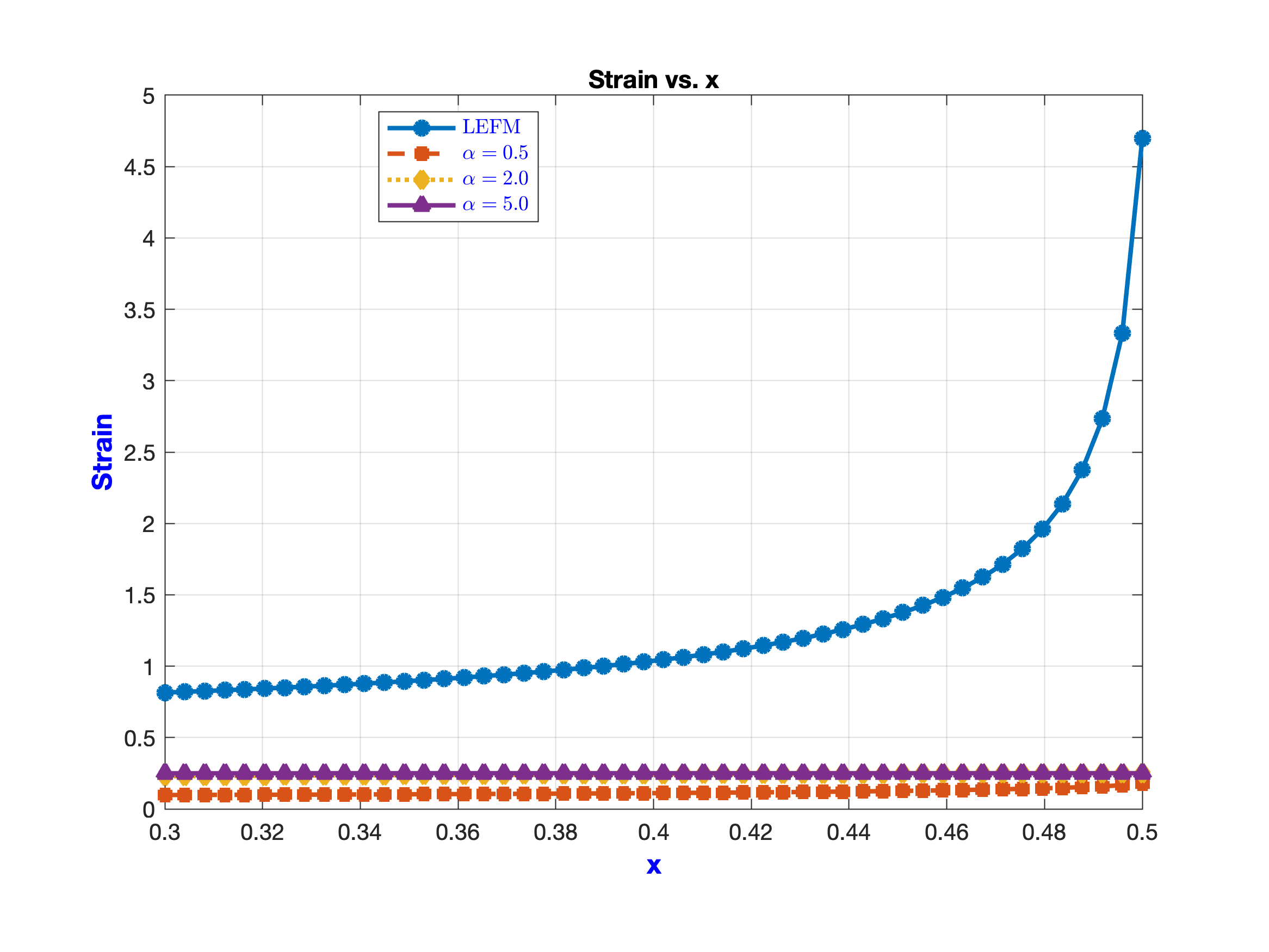}
        \caption{Effect of $\alpha$ on $\bfa{\epsilon}_{23}$}
        \label{fig:fig3}
    \end{subfigure}
    \hfill 
    \begin{subfigure}[b]{0.48\textwidth}
        \includegraphics[width=\linewidth]{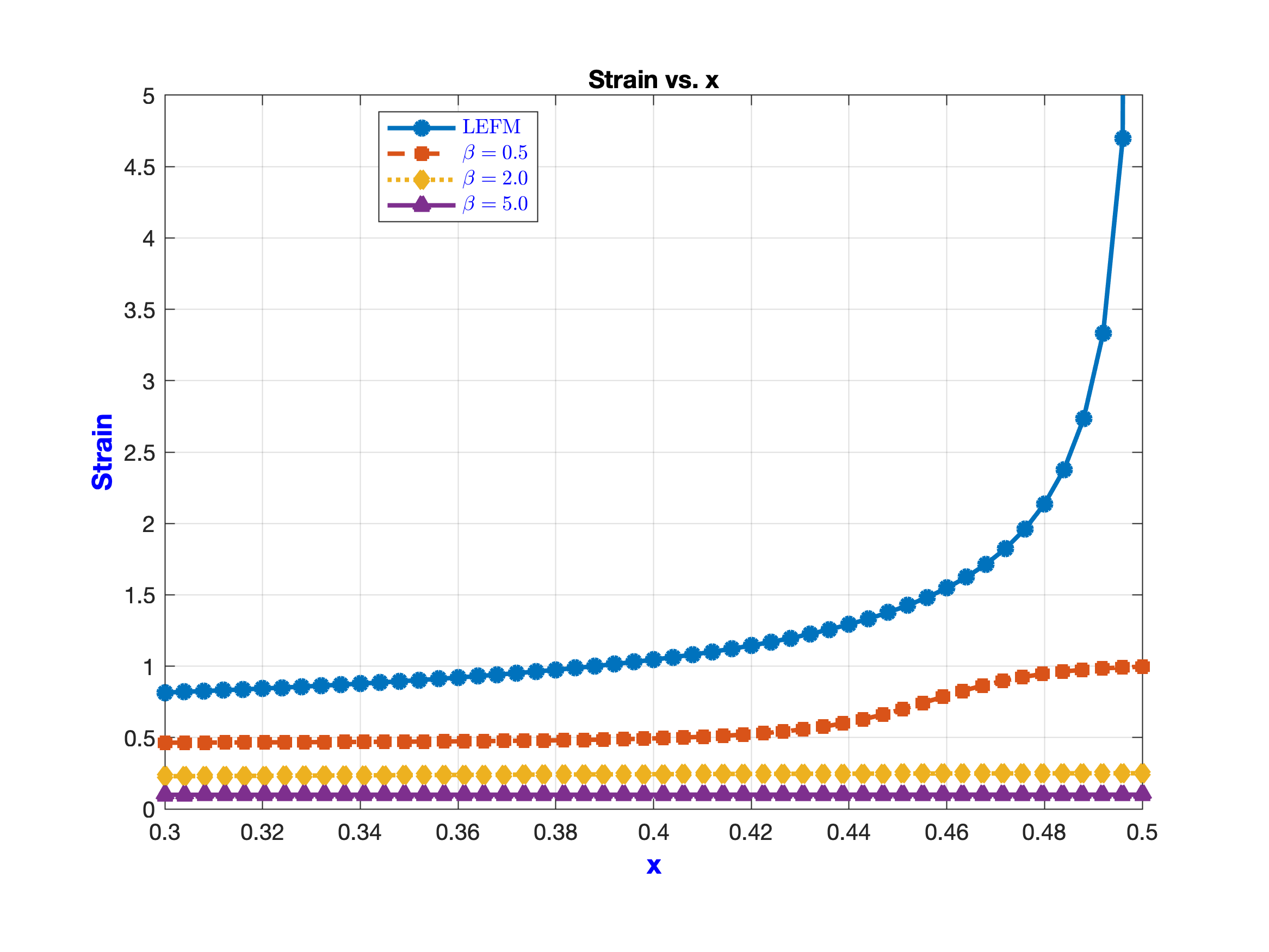}
        \caption{Effect of $\beta$ on $\bfa{\epsilon}_{23}$}
        \label{fig:fig4}
    \end{subfigure}
    \vspace{0.5cm} 
    \begin{subfigure}[b]{0.48\textwidth}
        \includegraphics[width=\linewidth]{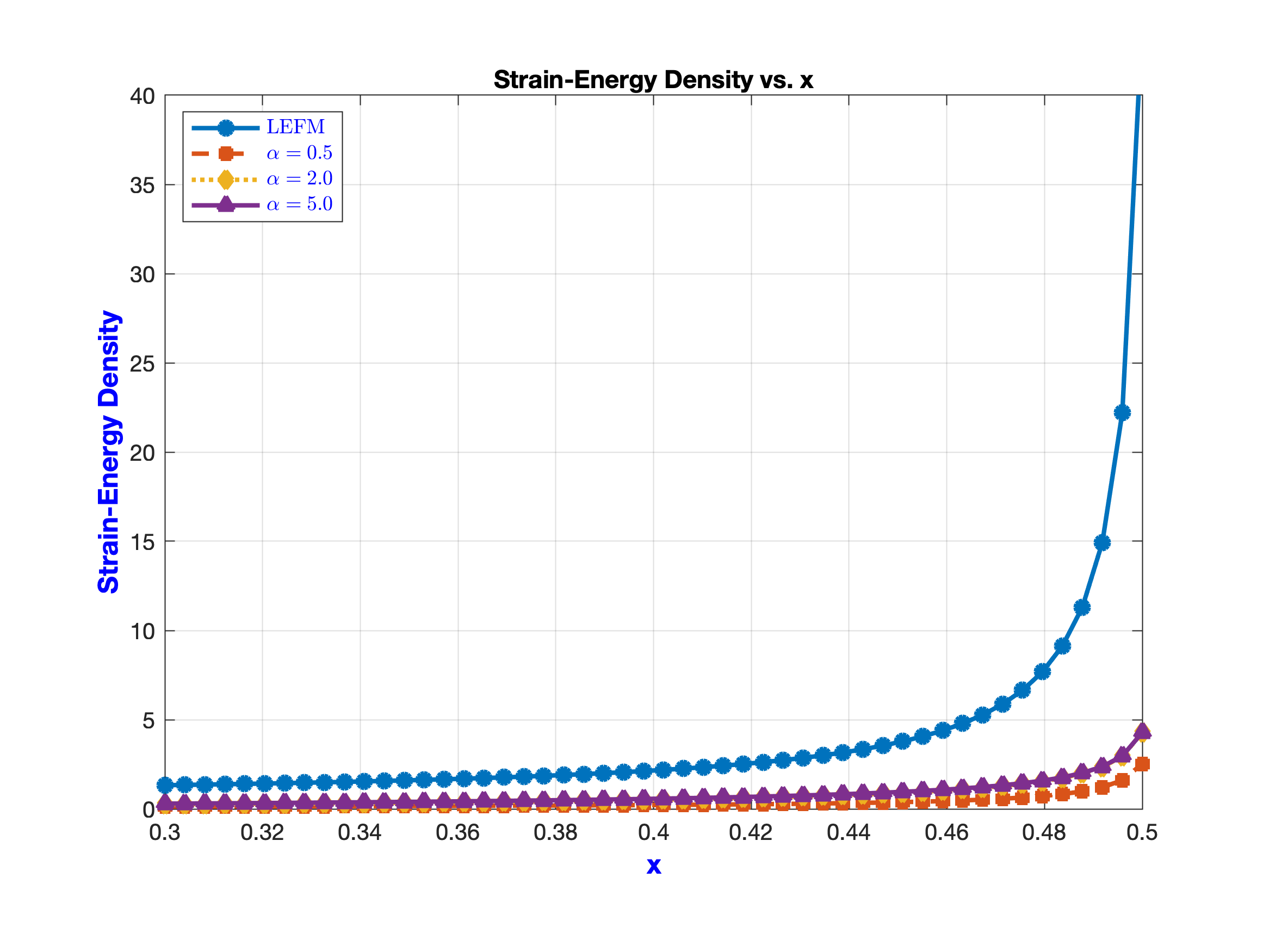}
        \caption{Effect of $\alpha$ on strain-energy density}
        \label{fig:fig5}
    \end{subfigure}
    \hfill 
    \begin{subfigure}[b]{0.48\textwidth}
        \includegraphics[width=\linewidth]{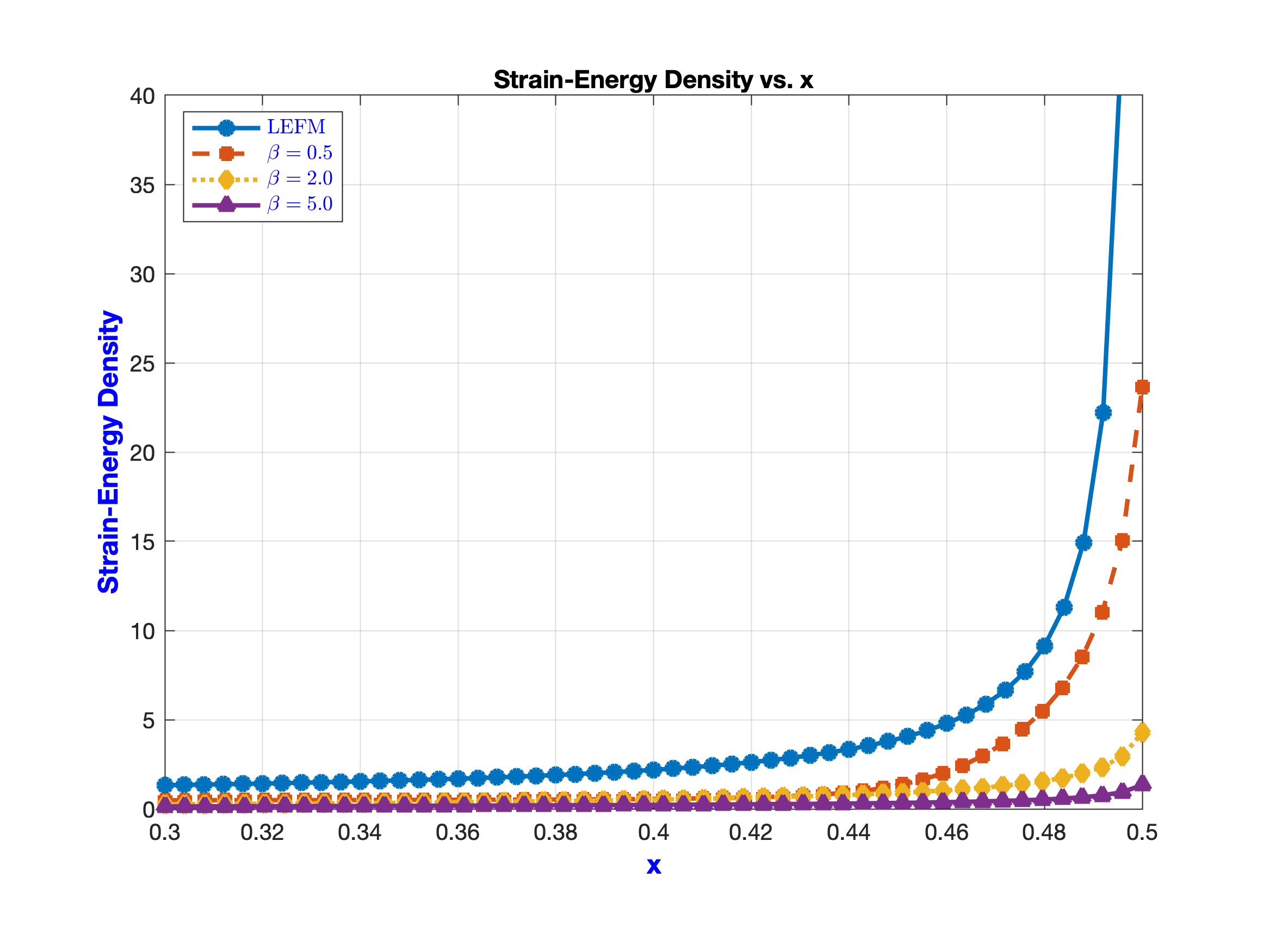}
        \caption{Effect of $\beta$ on strain-energy density}
        \label{fig:fig6}
    \end{subfigure}
\caption{This two-panel figure comprehensively illustrates the influence of two key modeling parameters, $\alpha$ and $\beta$, on the crack-tip fields. The left panel specifically focuses on the impact of $\alpha$ on critical crack-tip quantities, including stress, strain, and strain-energy density. Conversely, the right panel presents the effect of $\beta$ on these same crack-tip fields, providing a comparative analysis of each parameter's contribution.}
    \label{fig:crack_tip_fields}
\end{figure}

Figure~\ref{fig:crack_tip_fields} meticulously illustrates the impact of two critical modeling parameters, $\alpha$ and $\beta$, on the crack-tip fields. Specifically, it examines stress ($\bfa{T}_{23}$), strain ($\bfa{\epsilon}_{23}$), and strain-energy density ($\bfa{T} \cdot \bfa{\epsilon}$) along a reference line defined by $0.3 \leq x \leq 0.5$ and $y=0.5$, positioned directly ahead of the crack-tip. A clear observation is that stress concentration intensifies with increasing values of $\alpha$. Conversely, an opposite trend is observed for the parameter $\beta$, where increasing values appear to mitigate stress concentration. For both parameters, the crack-tip strain values deviate from the growth trend seen in the linear model (where $\beta=0$). Despite these variations, the strain-energy density consistently remains highest near the crack-tip when compared to values along the reference line. These findings are robust and align well with prior observations across various loading conditions, including anti-plane shear, plane-strain, in-plane shear, and even three-dimensional problems, as supported by existing literature \cite{mallikarjunaiah2015direct,manohar2024hp,vasilyeva2024generalized,ghosh2025finite,yoon2024finite,yoon2022MAM,gou2015modeling,ghosh2025computational,gou2023computational,gou2023finite,gou2025computational,mallikarjunaiah2025crack}. .

\section{Conclusion}\label{conclusions}

In this study, we successfully performed an $hp$-adaptive finite element simulation of a static anti-plane shear crack within a nonlinear strain-limiting elastic solid. This work demonstrates a robust approach to accurately capturing the complex stress, strain, and strain-energy density fields in the vicinity of a crack-tip, particularly within the framework of a constitutive model designed to address the inherent singularities predicted by classical linear elastic fracture mechanics. The constitutive relationship considered in this study is algebraically nonlinear, but the material body is geometrically linear. 

The application of $hp$-adaptive finite element methods proved instrumental in achieving the efficiency required for this challenging problem considered in this paper. It is now well-documented that $hp$-FEM achieves exponential rates of convergence by simultaneously refining the mesh ($h$-refinement) and increasing the polynomial order of the approximation ($p$-refinement). This dual refinement strategy is particularly effective for problems with localized features and strong singularities, such as crack-tips, where traditional $h$-refinement alone can be computationally prohibitive. By dynamically adapting both mesh size and polynomial degree, $hp$-FEM optimally concentrates computational effort where it is most needed, ensuring accurate resolution of the highly localized gradients without excessive global mesh density.

Furthermore, integrating the nonlinear strain-limiting elastic theory provided a physically sound framework to regularize the crack-tip fields. Unlike classical elasticity, which predicts infinite stresses and strains at a crack-tip, the strain-limiting theory introduces material parameters ($\alpha$ and $\beta$ in this context) that govern the material's response for any deformations, thereby yielding finite and more realistic values for stress, strain, and strain-energy density. Our simulations demonstrated how these parameters influence the crack-tip fields: increasing $\alpha$ was shown to intensify stress concentration, while $\beta$ exhibited an opposing effect. The observed deviations in crack-tip strain behavior compared to linear models and the consistent localization of strain-energy density near the crack-tip reinforce the importance of this nonlinear constitutive description. These findings are consistent with a broad range of previous observations across various loading and dimensional scenarios, further validating the model and the numerical approach.

In conclusion, this research highlights the synergistic power of $hp$-adaptive finite element methods and the strain-limiting elastic theory for accurately and efficiently analyzing fracture problems in nonlinear materials. The methodology presented here offers a robust tool for understanding complex material behavior near singularities, paving the way for more reliable predictions in engineering design and material science. Future work could extend this framework to quasi-static crack propagation \cite{manohar2025convergence,yoon2021quasi,lee2022finite,manohar2025adaptive}, fatigue analysis, or more complex three-field formulations \cite{fernando2025xi,fernando2025}, further leveraging the capabilities of $hp$-FEM for advanced nonlinear fracture mechanics.

\section*{Acknoledgement}
The work of SMM is supported by the National Science Foundation under Grant No.\ 2316905.

\bibliographystyle{plain}
\bibliography{references}

\end{document}